\documentclass{article}
\usepackage{lmodern}
\usepackage[scale=2]{ccicons}
\usepackage[frenchb,english]{babel}
\usepackage[T1]{fontenc}
\usepackage{multirow}
\usepackage{adjustbox}
\usepackage{amsmath}
\usepackage{bbm}
\numberwithin{equation}{section}
\usepackage{amsfonts}
\usepackage{amssymb}
\usepackage{xcolor}
\colorlet{linkequation}{cyan}
\usepackage[colorlinks]{hyperref}

\newcommand*{\refeq}[1]{%
  \begingroup
    \hypersetup{
      linkcolor=linkequation,
      linkbordercolor=linkequation,
    }%
    \ref{#1}%
  \endgroup
}
\usepackage{dsfont}
\usepackage{caption}
\usepackage{tikz}
\usepackage{amsthm}
\usepackage{color}
\usepackage{graphicx}
\usepackage{tikz}
\usetikzlibrary{automata,topaths}
\usetikzlibrary{arrows,positioning,shadows,fit,shapes}
\usepackage{enumitem}
\newtheorem*{remark}{Remark}

\newtheorem{theorem}{Theorem}[section]
\newtheorem{lemma}{Lemma}[section]
\newtheorem{proposition}{Proposition}[section]
\newtheorem{corollary}{Corollary}[section]
\newtheorem{definition}{Definition}[section]
\newcommand{\R}{\mathbb{R}}
\newcommand{\E}{\mathbb{E}}
\newcommand{\Z}{\mathbb{Z}}
\newcommand{\e}{e}
\newcommand{\F}{\mathcal{F}}
\renewcommand{\P}{\mathbb{P}}
\newcommand{\Q}{\mathbb{Q}}
\newcommand{\si}{\sigma}
\newcommand{\indic}{\mathbb{I}}

\title{Critical scaling for an anisotropic percolation system on $\Z^2$}
\author{Thomas Mountford \thanks{Institut de Math\'ematiques, Ecole Polytechnique F\'ed\'erale de Lausanne, 1015 Lausanne, Switzerland. Email: thomas.mountford@epfl.ch, hao.xue@epfl.ch} \and
Maria Eul\'alia Vares \thanks{Instituto de Matem\'atica, Universidade Federal do Rio de Janeiro, RJ, Brazil. Email: eulalia@im.ufrj.br} \and
Hao Xue \footnotemark[1]}

\begin{document}
\maketitle
\begin{abstract}
In this article, we consider an anisotropic finite-range bond percolation model on $\mathbb{Z}^2$. On each horizontal layer $\{(x,i)\colon x \in \mathbb{Z}\}$ we have edges $\langle (x,i),(y,i)\rangle$ for $1 \le |x-y|\le N$. There are also vertical edges connecting two nearest neighbor vertices on distinct layers $\langle (x,i),(x,i+1)\rangle$ for $x,i \in \mathbb{Z}$. On this graph we consider the following anisotropic independent percolation model: horizontal edges are open with probability $1/(2N)$, while vertical edges are open with probability $\epsilon$ to be suitably tuned as $N$ grows to infinity. The main result tells that if $\epsilon = \kappa N^{-2/5}$, we see a phase transition in $\kappa$: positive and finite constants $C_1, C_2$ exist so that there is no percolation if $\kappa <C_1$ while percolation occurs for $\kappa >C_2$. The question is motivated by a result on the analogous layered ferromagnetic Ising model at mean field critical temperature \cite[\emph{J. Stat. Phys.} {\bf 161}, (2015), 91--123]{eulalia} for which the authors showed the existence of multiple Gibbs measures for a fixed value of the vertical interaction and conjectured a change of behavior in $\kappa$ when the vertical interaction suitably vanishes as $\kappa\gamma^{b}$, where $1/\gamma$ is the range of the horizontal interaction. For the product percolation model we have a value of $b$ that differs from what was conjectured in that paper. The proof relies on the analysis of the scaling limit of the critical branching random walk that dominates the growth process restricted to each horizontal layer and a careful analysis of the true horizontal growth process, which is interesting by itself. This is inspired by works on the long range contact process \cite[\emph{Probab. Th. Rel. Fields} {\bf 102}, (1995), 519--545]{mueller}. A renormalization scheme is used for the percolative regime.
\end{abstract}

\section{Introduction}
\label{intro}
In this article, we consider an anisotropic finite-range bond percolation model on the plane. For this we let $\Z^2=(V,E)$ be the graph with vertex set $V=\{v=(x,i):x\in \Z,i\in \Z\}$ and edge set $E=E_N=\{ e=\langle v_1,v_2\rangle:v_k=(x_k,i_k), k=1,2;|i_1-i_2|=1 \text { for }x_1=x_2 \text{ and }1\leq |x_1-x_2|\leq N \text{ for }i_1=i_2 \}$. The edge set can be partitioned into two disjoint subsets $E=E_v \cup E_h$. $E_v$ is the set of vertical edges, s.t. $E_v=\{e=\langle v_1,v_2\rangle: x_1=x_2\}$ and $E_h$ denotes the set of horizontal edges, s.t. $E_h=\{e=\langle v_1,v_2\rangle:i_1=i_2\}$ (here $(x_k,i_k)$ corresponds to $v_k$ for $k=1,2$). Each vertical edge is open with probability $\epsilon$ and each horizontal edge is open with probability $1/(2N)$, and they are all independent of each other. Our main purpose is to study the existence of percolation in this system, with $\epsilon=\epsilon (N)$ that tends to zero as $N$ grows to infinity.

The basic motivation for this paper comes from a question raised in \cite{eulalia}, where the authors investigated the existence of phase transition for an anisotropic Ising spin system on the square lattice $\mathbb{Z}^2$. On each horizontal layer $\{(x,i) \colon x \in \mathbb{Z}\}$ the $\{-1,+1\}$-valued spins $\sigma(x,i)$ interact through a ferromagnetic Kac potential at the mean field critical temperature, i.e. the interaction between the spins $\sigma(x,i)$ and $\sigma (y,i)$ is given by
\begin{equation*}
- J_\gamma(x,y) \si(x,i)\si(y,i),\quad \sum_{y\ne x} J_\gamma(x,y)=1,
\end{equation*}
where $J_\gamma(x,y) = c_\gamma \gamma J(\gamma (x-y))$, and one assumes  $J(r)$, $r\in \mathbb R$,
to be smooth and symmetric with  support in $[-1,1]$, $J(0)>0$, $\int J(r)dr=1$, and moreover $c_\gamma$ is the normalization constant ($c_{\gamma}\rightarrow 1$ as $\gamma\rightarrow 0$). To this one adds a small nearest neighbor vertical interaction
$$
-\epsilon \sigma(x,i)\sigma (x,i+1),
$$
and the authors proved in \cite{eulalia} that given any $\epsilon >0$, for all $\gamma>0$ small $\mu^+_\gamma \neq \mu^-_\gamma$, where $\mu^+_\gamma$ and $\mu^-_\gamma$ denote the Dobrushin-Lanford-Ruelle (DLR) measures obtained as thermodynamic limits of the Gibbs measures with $+1$, respectively $-1$ boundary conditions.

One of the questions left open in \cite{eulalia} has to do with the following: how small can we take $\epsilon=\epsilon(\gamma)$ and still observe a phase transition of the Ising model (for all $\gamma$ small)? Following various considerations, the authors conjectured that if $\epsilon=\epsilon(\gamma)=\kappa \gamma^{2/3}$ we might see a different behavior while varying $\kappa$. This is the problem that motivates this paper. Our technique does not give an answer to the Ising system, but considering the related product percolation model and usual Fortuin-Kasteleyn-Ginibre (FKG) comparison, it yields a partial answer to the question, and shows that the conjecture has to be modified. (See Remark (b) after the statement of Theorem \ref{thm:main}.)

Indeed, the original problem just described could be formulated in terms of percolation for a Fortuin-Kasteleyn (FK) measure with shape parameter $q=2$. Here we treat a simpler case by considering a corresponding anisotropic percolation model on $\Z^2$. Since a Kac potential can be taken as an interaction of strength $\gamma$ with range $\gamma^{-1}$ (which we fix as $N$) we consider the edge percolation problem where horizontal edges of length at most $N$ are open with probability $1/(2N)$ and the vertical edges between sites at distance $1$ are open with probability $\epsilon=\kappa N^{-b}$.

With respect to layer $0$, we denote $\mathcal{C}_{\mathbf{0}}^0$ as the cluster containing $(0,0)$, $$\mathcal{C}_{\mathbf{0}}^0=\{x:(0,0)=\mathbf{0}\rightarrow (x,0) \text{ with all the edges along the path in }\Z\times\{0\}\},$$
where $v_1\rightarrow v_2$ means there is an open path from $v_1$ to $v_2$. We can speak of generations on each horizontal layer (we will consider the horizontal behaviour on layer $0$ for simplicity). $x\in \mathcal{C}_{\mathbf{0}}^0$ is of $k$-th generation if the shortest open path from $(0,0)$ to $(x,0)$ has graph distance $k$. That means that there are vertices $v'_1,\cdots,v'_k$ such that $v'_1=(0,0),v'_k=(x,0)$ and for any $1\leq i\leq k-1$,$\langle v'_i,v'_{i+1}\rangle\in E_h$ is open. Denote $G_k^0$ as the collection of vertices that can be reached from $\mathbf{0}$ at $k$-th generation. The sites of $\{G_k^0\}_{k\geq 0}$ form a process very close to a  branching random walk starting from $0$. The difference between $\{G^0_k\}_{k\geq 0}$ and a critical branching process is the domain of the state function. Denote $\{\xi_k(x)\}_{k\geq 0}$ as the critical branching random walk. At each time $k$, particles of occupied sites branch following $\text{Binomial}(2N,1/(2N))$ and move to its $2N$ neighbours uniformly (note that it is different from the model in \cite{lalley} discussed in the next paragraph). The state function $\xi_k(x)\in \Z_+$. However, the process $\{G^0_k\}_{k\geq 0}$ only tells if the site is occupied or not, which takes value in $\{0,1\}$. In Section \ref{sec:envelope}, we show that these two processes are not too different.  This motivates us to consider the asymptotic density on each horizontal layer and use it to derive the cumulated occupied sites over generations. But the introduction of generations will cause a problem in the percolation problem if we only consider the branching random walk. Because we are interested in percolation, the vertical connections should be considered only once over the generations. Therefore, the true process we are considering is a branching random walk with attrition. The attrition means that if any site has been visited during the propagation, it cannot be visited again. 

The way of dealing with horizontal propagation is motivated by the work of Lalley \cite{lalley} on the scaling limit of spatial epidemics on the one-dimensional lattice $\Z$ to Dawson-Watanabe process with killings. The process considered in \cite{lalley} is as follows. At each site $i$, there is a fixed population (or village) of $N$ individuals and each of them can be either susceptible, infected or recovered. The model runs in discrete time; an infected individual recovers after a unit of time and cannot be infected again. An infected individual may transmit the infection to a randomly selected (susceptible) individual in the same or in the neighboring villages. Denote $p_N(i,j)$ as the transmission probability between any infected particle at site $i$ and any susceptible particle at site $j=i+e$, where $e=0$ or $\pm1$. For any pair $(x_i,u_j)$ of infected and susceptible individuals located at $i$ and $j$ respectively with $|i-j|\leq 1$, the transmission probability is taken as $$p_N(i,j)=\frac{1}{3N},$$ which makes it asymptotically critical (as $N\rightarrow\infty$).
The evolution of this SIR dynamics can be studied with the help of a branching random walk envelope: any individual at site $i$ and time $t$ lives for one unit and reproduces, placing a random number of individuals at a nearest site $j$ with $|j-i|\leq 1$, where the random number is of law $\text{Binomial}(N,1/(3N))$. The individuals are categorized into Susceptible, Infected or Recovered (SIR) and any recovered individual is immune and will not be infected again. The author studied the scaling limit (space factor $N^{\beta/2}$ and time factor $N^{\beta}$) of this system by considering the cluster of particles at each site village and calculating the log-likelihood functions. The recovered individuals do matter only when $\beta=2/5$, which corresponds to the attrition part of our process. To study the scaling limit of our process on horizontal level, we first need to perform space and time rescaling on the approximate density. First we have to scale the space with $N$, then the movement of the edges from $x$ will have a uniform displacement on $x+[-1,1]/\{0\}$.  Then, to get the weak convergence, we will renormalize the space and time with $N^{\alpha}$ and $N^{2\alpha}$ respectively. The state of the process at time $n\in \Z^+$ is given by $\hat{\xi}_n(\cdot):\Z/N^{1+\alpha}\rightarrow \{0,1\}$. $\hat{\xi}_n(x)=0$ indicates that the site $x$ is vacant and $\hat{\xi}_n(x)=1$ indicates that the site $x$ is occupied. Two sites are neighbors in the scaled space, denoted by $y\sim x$ if $|x-y|\leq N^{-\alpha}$ (or $j\sim i$ if $|j-i|\leq N$ in the unrescaled space). We are going to consider the asymptotic approximate density $$(A\hat{\xi})(x)=\frac{1}{2N^{\alpha}}\sum_{y\sim x}\hat{\xi}(y),$$
and study its limit after the above mentioned time change.

The method in \cite{lalley} is to calculate the log-likelihood function with respect to a branching envelope with known asymptotic density. However, we do not have the log-likelihood function in our case. A more standard argument is to show the weak convergence of the rescaled continuous-time particle system by verifying the tightness criteria \cite{ek} like in \cite{mueller}, \cite{rescaledvoter} and \cite{rescaledcontact}. We will mainly refer to the way of Mueller and Tribe \cite{mueller} dealing with long-range contact process and long-range voter model and adapt it to our discrete model to get the asymptotic stochastic PDEs. Our strategy on the horizontal level is to derive the asymptotic density of the branching random walk without attrition dominating the true system, where the states are denoted by $\xi(x)$. In the branching random walk, the case of multiple particles at one site is allowed. But we can show that the probability of multiple particles is small with order $O(N^{\alpha-1})$. Then the state can be reduced from $\mathbb{N}$-valued to $\{0,1\}$-valued. We will then derive the asymptotic density of the true process. Since we are considering the existence of percolation, to consider the infinite cluster containing $(0,0)$ is equivalent to consider $2\lfloor N^{2\alpha}\rfloor$ equally spaced particles on $\left\{-\lfloor N^{1+\alpha}\rfloor,\dots,0,\dots, \lfloor N^{1+\alpha}\rfloor\right\}$ (so the distance between particles in $\Z$ is of order $N^{1-\alpha}$). Indeed, if we denote $[-r,r]_N=[-r,r]\cap\Z/N^{1+\alpha}$ as the rescaled discrete interval, to show percolation we may take an initial condition $\hat\xi_0$ with finite support, such that $A(\hat\xi_0)(x)=1$ for $x\in[-1,1]_r$ and whose linear interpolation tends (as $N\rightarrow\infty$) to a continuous function $f$ with compact support such that $f(x)=1$ for $x\in[-1,1]$. For simplicity, we may take $f$ to vanish outside $[-1-\delta,1+\delta]$ for some $\delta>0$ fixed, and linear in $[-1-\delta,1]$ and $[1,1+\delta]$.

When showing percolation and adding the vertical connections, by a renormalization argument (ref. \cite{supercriticalcontact}) we can reduce our layered system to an oriented percolation. We can define a site as open if its corresponding block has a certain amount of cumulated density, since we have already taken into account the attrition in the true system. As it will be explained in Section \ref{sec:existence}, we will indeed need to consider the density on a smaller scale. After building the renormalization argument, we are able to use the criteria in \cite{durrett} to determine the existence of percolation. The main result of this article is as follows.

\begin{theorem}
\label{thm:main}
The critical values of the scaling and interaction factors are $b=2\alpha=2/5$. That is there exist positive constants $C_1$ and $C_2$ such that for $\kappa<C_1$ not depending on $N$, there is no percolation and for $\kappa>C_2$, there is a percolation, where $\kappa N^{-b}$ is the opening probability of vertical edges.
\end{theorem}

\begin{remark}
(a) The critical value $\alpha=1/5$ can be guessed by standard coupling as in \cite{lalley}.  First, we build a critical branching random walk with the same initial conditions, namely $2\lfloor N^{2\alpha}\rfloor$ particles with at most one on each site distributed uniformly on $2\lfloor N^{1+\alpha}\rfloor$ sites in $[-1,1]_N$. At each time, particles of the branching random walk produce offsprings at their neighbourhoods following a Binomial distribution $\text{Bin}(2N,1/(2N))$. The branching random walk will finally become extinct as we know \cite{branching}. The existence of percolation is a meaningful problem if we introduce a vertical interaction. In the beginning, there are $O(N^{\alpha-1})$ particles at each site on average. Since the branching random walk is critical, i.e. the expectation of offspring is exactly $1$, this average behaviour will not change too much during the propagation.
Next, we colour the particles as red or blue according to they are alive or dead respectively. The attrition means that if a site has been visited, then it cannot be visited again. Initially, all the particles are red. The offspring of blue particles are blue and the choice of colour of red particles is as follows. If a site $x$ has been occupied  in the past, then the offsprings of red particles that are produced at $x$ become blue. The branching random walk will last for $O(N^{2\alpha})$ generations (ref. \cite{branching}). Up to extinction, the chance of dying for any particle is $O(N^{3\alpha-1})$. The total attrition at each generation is $O(N^{5\alpha-1})$. Hence if $\alpha=1/5$, then the total attrition per generation is $O(1)$.

(b) As already mentioned above, the original problem that motivated this paper can be formulated in terms of the existence of percolation for a corresponding Fortuin-Kasteleyn measure with shape parameter $q=2$ and edge probabilities of $\{\langle v_1,v_2\rangle\in E,v_1=(x,i),v_2=(y,j)\}$ to be
$$p(\langle v_1,v_2\rangle)=1-e^{-J_{\gamma}(x,y)}\indic_{\{\langle v_1,v_2\rangle\in E_h\}}-e^{\epsilon(\gamma)}\indic_{\{\langle v_1,v_2\rangle\in E_v\}}.$$
By the FKG inequality, the probability of percolation for $q=2$ is bounded from above by that when $q=1$ (product measure). As a consequence, for $\kappa$ sufficiently small, we conclude that there is no phase transition if $\epsilon(\gamma)= \kappa \gamma^{2/5}$ (much larger than $\kappa\gamma^{2/3}$ as conjectured in \cite{eulalia}), for all $\gamma$ small.

(c) The organization of this paper is as follows. We will provide necessary lemmas and use them to show the weak convergence of the dominating envelope in Section \ref{sec:envelope}.  Results related to the true horizontal process like the asymptotic density, Girsanov transformation and cumulated density are shown in Section \ref{sec:real}. The killing property of the attrition part helps us to consider the case when $\kappa<C_1$ in Subsection \ref{sec:sub}. With the properties of the true process, the oriented percolation construction is built up in Subsection \ref{sec:sup} and we can show the existence of percolation when $\kappa>C_2$.
\end{remark}

\section{The envelope process}
\label{sec:envelope}

Since the proof of our main theorem involves the scaling limit of the process $\hat\xi$, as in \cite{lalley} we are led to consider first the situation of the corresponding branching random walk, as our 'envelope process', and we first study its scaling limit in Theorem \ref{thm: brw}. This result should probably be contained in the literature, even if not stated exactly as convenient for the consideration of our true model in Section \ref{sec:real}; the ideas are contained in Mueller and Tribe \cite{mueller}.

Before studying the asymptotic behaviour of the process, we first study that of an envelope process. In this section, we consider the state function $\xi_n(\cdot):\Z/N^{1+\alpha}\rightarrow \Z_+$. The mechanism of this envelope process is as follows. The number of particles at site $x$ will increase by $1$ if one of its neighbours branches following $\text{Binomial}(2N,1/(2N))$ and then chooses $x$ uniformly among the $2N$ neighbours. It can be written as
$$\xi_{n+1}(x)=\sum_{y\sim x}\sum_{w=1}^{\xi_n(y)}\eta^w_{n+1}(y,x),$$ where $(\eta^w_{n+1}(y,x))_{w,n,y,x}$ is an i.i.d. sequence with distribution $\text{Bernoulli}(1/(2N))$.
The horizontal process $\hat{\xi}_n(\cdot):\Z/N^{1+\alpha}\rightarrow\{0,1\}$ analysed in Section \ref{sec:real} is dominated by this envelope process in two senses: $\hat{\xi}_n(\cdot)$ does not allow multiple particles at one site and any visited site cannot be visited again.
At the end of this section, we will show that the probability of multiple particles at one site is quite small, of order $O(N^{2(\alpha-1)})$ which is negligible when $\alpha<1$. 

The main result of this section regards the asymptotic behaviour (as $N\rightarrow\infty$) of the approximate density function of the dominating envelope process
$$A(\xi_{\lfloor tN^{2\alpha}\rfloor})(x)=\frac{1}{2N^{\alpha}}\sum_{y\sim x}\xi_{\lfloor tN^{2\alpha}\rfloor}(y)$$
extended to $\R$ as the linear interpolation of its values on $\Z/N^{1+\alpha}$. This is made precise in Theorem \ref{thm: brw} below.
\begin{remark}
The same interpolation is used when considering the approximate density of the process $\hat\xi$.
\end{remark}
Setting $\e_{\lambda}(x)=\e^{\lambda |x|}$ for $\lambda\in\R$, we define
$$\mathcal{C}=\left\{f:\R\rightarrow[0,\infty)\text{ continuous with }|f(x)e_{\lambda}(x)|\rightarrow0\text{ as }|x|\rightarrow\infty,\forall \lambda<0\right\},$$
to which we give the topology induced by the norms $(\|\cdot\|_{\lambda},\lambda<0)$,
where $$\| f\|_{\lambda}=\sup_x|f(x)e_{\lambda}(x)|.$$
In the following convergences (Theorem \ref{thm: brw} and Theorem \ref{thm real spde}), we consider the law of $A(\xi)$ or $A(\hat{\xi})$ in the space $D([0,\infty),\mathcal{C})$, the space of $\mathcal{C}$-valued paths equipped with Skorohod topology.

\begin{theorem}
\label{thm: brw}
Assume that as $N\rightarrow\infty$, $A(\xi_0)$ converges in $\mathcal{C}$ to a continuous function $f$ with compact support.
Then, $A(\xi_{\lfloor tN^{2\alpha}\rfloor})(x)$ converges in law to $u_t(x)$, which is the solution to one dimensional Dawson-Watanabe process:
\begin{align}\label{eqn: brw spde}
\begin{cases}
\frac{\partial u_t}{\partial t}=\frac{1}{6}\Delta u_t+\sqrt{u_t}\dot{W}(t,\cdot) \\
u_0=f,
\end{cases}
\end{align}
where $\Delta$ is the Laplacian operator acting in the spatial coordinates and $\dot{W}$ is the space-time white noise.
\end{theorem}
The idea of the proof is to write the mechanism as a martingale problem, then introduce a Green function representation (see (\ref{approx})) to simplify the approximate density. The tightness criteria in \cite{ek} can be applied to get the weak convergence. We will follow the blueprint of \cite{mueller} to show the tightness.

Before starting the proof, we first explain the notation used in the following sections. For $f,g$ functions on our discrete space $\Z/N^{1+\alpha}$, we write, whenever meaningful,
$$(f,g)=\frac{1}{N^{1+\alpha}}\sum_xf(x)g(x)$$
Similarly, for  $f,g$ defined in $\R$, we write
$$(f,g)=\int fgdx.$$
Define the discrete measure generated by $\xi_n$ as
$$\nu_n^N=\frac{1}{N^{2\alpha}}\sum_x\xi_n(x)\delta_x$$
and for a function $f$ and measure $\nu$, we write
$$(f,\nu)=\int fd\nu$$for the integral whenever it is well defined.
In Lemma \ref{measure density}, we will see that for any test function $f$ which is bounded and with compact support
$$(f,A\xi_n)-(f,\nu_n^N)\rightarrow 0 \text{ in }L^2.$$
We define the amplitude of a function around a neighbourhood as 
\begin{equation}
\label{eqn:amp}
D(f,\delta)(x)=\sup \{ |f(y)-f(x)|:|y-x|\leq \delta \}.
\end{equation}

\subsection{Martingale Problem}
\label{green fun}

Suppose $\xi_0(x)$ is deterministic and with a finite (depending on $N$) support.
Rewriting the mechanism of $\xi_n(x)$, we have
\begin{equation}
\label{brw mech}
\begin{split}
\xi_{n+1}(x) &= \sum_{y\sim x}\sum_{w=1}^{\xi_n(y)}\left(\eta^w_{n+1}(y,x)-\frac{1}{2N}\right)+\frac{1}{2N}\sum_{y\sim x}\xi_n(y)\\
 &=\sum_{y\sim x}\sum_{w=1}^{\xi_n(y)}\left(\eta^w_{n+1}(y,x)-\frac{1}{2N}\right)+\frac{1}{2N}\sum_{y\sim x}(\xi_n(y)-\xi_n(x))+\xi_n(x).
\end{split}
\end{equation}
The first term will contribute to the space-time white noise part and the second term will contribute to the Laplacian in the SPDE.

Take discrete test function $\phi^N(k,x)$ for $x\in\Z/N^{1+\alpha}$ ($\phi^N_k(x)=\phi^N(k,x)$). $\phi^N(k,x):\mathbb{N}\times\Z/N^{1+\alpha}$ satisfying the following conditions:
\begin{equation}
\label{test fun cond}
\begin{split}
&\sum_{k=1}^{\lfloor TN^{2\alpha}\rfloor}(|\phi^N_k-\phi^N_{k-1}|,1)<\infty, \\
&\frac{1}{\lfloor TN^{2\alpha}\rfloor}\sum_{k=1}^{\lfloor TN^{2\alpha}\rfloor}(|\phi^N_k|+|\phi^N_k|^2,1)<\infty.
\end{split}
\end{equation}

Summation by parts and (\ref{brw mech}) give
\begin{align*}
(\nu_n^N,\phi^N_n) &= \frac{1}{N^{2\alpha}}\sum_x \xi_n(x)\phi^N_n(x)\\
 &=\frac{1}{N^{2\alpha}}\sum_x \xi_n(x)(\phi^N_n(x)-\phi^N_{n-1}(x))+\frac{1}{N^{2\alpha}}\sum_x\xi_n(x)\phi^N_{n-1}(x)\\
 &=(\nu_n^N,\phi^N_n-\phi^N_{n-1})+(\nu_{n-1}^N,\phi^N_{n-1})+\frac{1}{N^{2\alpha}}\sum_x\frac{1}{2N}\sum_{y\sim x}(\xi_{n-1}(y)-\xi_{n-1}(x))\phi^N_{n-1}(x)\\
 &~~~~+\frac{1}{N^{2\alpha}}\sum_x\sum_{y\sim x}\sum_{w=1}^{\xi_{n-1}(y)}\phi^N_{n-1}(x)\left( \eta^w_n(y,x)-\frac{1}{2N} \right),
\end{align*}
where we use the decomposition (\ref{brw mech}) in the last equality.
Denote $\Delta_D f(x)=\frac{N^{2\alpha}}{2N}\sum_{y\sim x}(f(y)-f(x))$. By summation by parts into the second term again, we can obtain
\begin{align*}
(\nu_n^N,\phi^N_n)-(\nu_{n-1}^N,\phi^N_{n-1}) &= (\nu_n^N,\phi^N_n-\phi^N_{n-1})+(\nu_{n-1}^N,N^{-2\alpha}\Delta_D \phi^N_{n-1})+\\
   &~~~~+\frac{1}{N^{2\alpha}}\sum_x\sum_{y\sim x}\sum_{w=1}^{\xi_{n-1}(y)}\phi^N_{n-1}(x)\left( \eta^w_n(y,x)-\frac{1}{2N} \right)\\
   &=(\nu_n^N,\phi^N_n-\phi^N_{n-1})+(\nu_{n-1}^N,N^{-2\alpha}\Delta_D\phi^N_{n-1})+d_n(\phi^N),
\end{align*}
where
\begin{equation}
\label{eqn:d_n}
d_n(\phi^N)=\frac{1}{N^{2\alpha}}\sum_x\sum_{y\sim x}\sum_{w=1}^{\xi_{n-1}(y)}\phi^N_{n-1}(x)\left( \eta^w_n(y,x)-\frac{1}{2N} \right).
\end{equation}

Summing up $n$ from $1$ to $m$, we get a semimartingale decomposition
\begin{equation}\label{decomp}
(\nu_m^N,\phi^N_m)-(A\xi_0,\phi^N_0)=(\nu_m^N,\phi^N_m-\phi^N_{m-1})+\sum_{i=1}^{m-1} (\nu_{i}^N,\phi^N_i-\phi^N_{i-1}+N^{-2\alpha}\Delta_D\phi_{i})+M_{m}(\phi^N),
\end{equation}
where we use the identity $(\nu^N,N^{-2\alpha}\Delta_D \phi^N+\phi^N)=(A\xi,\phi^N)$.

$M_m(\phi^N)=\sum_{k=1}^{m} d_{k}(\phi^N)$ is a martingale with square variation
\begin{equation}
\label{bracket}
\begin{split}
\langle M(\phi^N)\rangle_m &= \sum_{k=1}^m\E_{k-1}d_k^2\\
&=\sum_{k=1}^m\frac{1}{2N^{1+4\alpha}}\sum_x\sum_{y\sim x}\xi_{k-1}(y)(\phi^N_{k-1})^2(x)\left(1-\frac{1}{2N}\right)\\
&\leq \sum_{k=1}^m\frac{\|\phi^N_{k-1}\|_0}{2N^{1+4\alpha}}\sum_x\sum_{y\sim x}\xi_{k-1}(y)\phi^N_{k-1}(x)\\
&=\sum_{k=1}^m \frac{\|\phi^N_{k-1}\|_0}{N^{2\alpha}}(A\xi_{k-1},\phi^N_{k-1}).
\end{split}
\end{equation}
For any $x\in \Z/N^{1+\alpha}$, let $\psi_i^z(x)\geq 0$ be the solution to
\begin{align*}
\begin{cases}
\psi_i^z-\psi_{i-1}^z=N^{-2\alpha}\Delta_D \psi_{i-1}^z \\
\psi_0^z(x)=\frac{N^{\alpha}}{2}I(x\sim z).
\end{cases}
\end{align*}
The solution of this equation is $\psi_n^z=N^{1+\alpha}\P(S_{n+1}=x-z)$, where $S_n=\sum_{i=1}^n Y_i$, with $(Y_i)$ i.i.d. uniformly distributed on $\{i/N^{1+\alpha},|i|\leq N\}$.

$\Delta_D$ can be seen as the generator of this symmetric random walk $S_n$ with steps of variance $\frac{c_3}{3}N^{-2\alpha}$ and $\E [Y^4]=\frac{c_4}{5N^{4\alpha}}$, where $c_3(N),c_4(N)\rightarrow 1$. $\psi_t^z(x)$ behaves asymptotically as $p(\frac{c_3t}{3N^{2\alpha}},z-x)$ (ref. Lemma \ref{lem: diff}),where $p(t,x)$ is the Brownian transition probability.

We apply (\ref{decomp}) with test function $\phi^N_k=\psi_{n-k}$ for $k\leq n-1$, so that the first drift term vanishes and $(\nu^N_n,\phi^N_n)=(\nu^N_n,\psi_0^x)=A(\xi_n)(x)$. Thus we obtain an approximation
\begin{equation}
\label{approx}
A(\xi_n)(x)=(\nu^N_0,\psi_n^x)+M_n(\psi_{n-\cdot}^x),
\end{equation}

where 
$M_n(\psi_{n-\cdot}^x) =\sum_{k=1}^nd_k(\psi_{n-k}^x)$ and $d_k(\phi^N)$ is as in (\ref{eqn:d_n}).
Proving the tightness of $A(\xi_{\lfloor tN^{2\alpha}\rfloor})$ is equivalent to prove that of $M_{\lfloor tN^{2\alpha}\rfloor}$. Some estimations on $\psi_n$ and the moments of $A(\xi_n)$ used to show the tightness are stated in the appendix \ref{appendix}. We skip the proofs which are very similar to those in \cite{mueller}.

\subsection{Tightness}
\label{tightness}
In this section, we assume an initial condition so that the linear interpolation of $A(\xi_0)$ converges to $f$ under $\|\cdot\|_{-\lambda}$ for any $\lambda>0$. To get the centred approximated density, by (\ref{approx}), let
$$\hat{A}(\xi_n)(x)=A(\xi_n)(x)-(\nu_0^N,\psi_n^x).$$
\begin{lemma}
\label{tight}
For $0\leq s\leq t\leq T$, $x,y\in \mathbb{Z}/N^{1+\alpha}$, $|t-s|\leq 1$, $|x-y|\leq 1$, $\lambda>0$ and $p\geq 2$,
\begin{equation}
\label{eqn: tightness}
\E |\hat{A}(\xi_{\lfloor tN^{2\alpha}\rfloor})(x)-\hat{A}(\xi_{\lfloor sN^{2\alpha}\rfloor})(y)|^p\leq C(\lambda,p,f,T)\e_{\lambda p}(x)\left( |x-y|^{\frac{p}{4}}+|t-s|^{\frac{p}{4}}+N^{-\frac{\alpha p}{2}} \right).
\end{equation}
\end{lemma}
\begin{proof}
We decompose this difference into space difference $\hat{A}(\xi_{\lfloor tN^{2\alpha}\rfloor})(x)-\hat{A}(\xi_{\lfloor tN^{2\alpha}\rfloor})(y)$ and time difference $\hat{A}(\xi_{\lfloor tN^{2\alpha}\rfloor})(y)-\hat{A}(\xi_{\lfloor sN^{2\alpha}\rfloor})(y)$. First, we deal with the space difference. The  Burkholder-Davis-Gundy (BDG) inequality (discrete version recalled in the appendix \ref{appendix}) gives that
\begin{align*}
\E |\hat{A}(\xi_{\lfloor tN^{2\alpha}\rfloor})(x)-\hat{A}(\xi_{\lfloor tN^{2\alpha}\rfloor})(y)|^p\leq \E \langle M(\psi_{\lfloor tN^{2\alpha}\rfloor-\cdot}^x-\psi_{\lfloor tN^{2\alpha}\rfloor-\cdot}^y)\rangle_{\lfloor tN^{2\alpha}\rfloor}^{\frac{p}{2}}.
\end{align*}
The constants $C(\lambda,p,f,T)$ in the following proof are generic constants. With a similar argument as in (\ref{bracket}),

\begin{align*}
&~~~~\langle M(\psi_{\lfloor tN^{2\alpha}\rfloor-\cdot}^x-\psi_{\lfloor tN^{2\alpha}\rfloor-\cdot}^y)\rangle_{\lfloor tN^{2\alpha}\rfloor} \\
&\leq \sum_{k=1}^{\lfloor tN^{2\alpha}\rfloor}\frac{\|\psi_{\lfloor tN^{2\alpha}\rfloor-k+1}^x-\psi_{\lfloor tN^{2\alpha}\rfloor-k+1}^y\|_{\lambda}}{N^{2\alpha}}\left(A\xi_{k-1}e_{-\lambda},\psi_{\lfloor tN^{2\alpha}\rfloor-k+1}^x+\psi_{\lfloor tN^{2\alpha}\rfloor-k+1}^y\right)
\end{align*}
By Lemma \ref{psi} (d),
\begin{align*}
&~~~~\langle M(\psi_{\lfloor tN^{2\alpha}\rfloor-\cdot}^x-\psi_{\lfloor tN^{2\alpha}\rfloor-\cdot}^y)\rangle_{\lfloor tN^{2\alpha}\rfloor}\\
&\leq \sum_{k=1}^{\lfloor tN^{2\alpha}\rfloor}\left((\lfloor tN^{2\alpha}\rfloor-k+1)^{-\frac{1}{2}}N^{-\alpha}|x-y|^{\frac{1}{2}}+N^{-\frac{3\alpha}{2}}k^{-\frac{3}{4}} \right)\cdot \\
&\qquad \qquad \qquad \qquad \qquad \qquad \qquad \qquad \left(A\xi_{k-1}e_{-\lambda},\psi_{\lfloor tN^{2\alpha}\rfloor-k+1}^x+\psi_{\lfloor tN^{2\alpha}\rfloor-k+1}^y\right).
\end{align*}

By using Lemma \ref{psi} (b) and Lemma \ref{moment} (c),
\begin{equation*}
\begin{split}
&~~~~\E |\hat{A}(\xi_{\lfloor tN^{2\alpha}\rfloor})(x)-\hat{A}(\xi_{\lfloor tN^{2\alpha}\rfloor})(y)|^p \\
&\leq C(\lambda,p,f,T)\e_{\lambda p}(x)\left( \sum_{k=1}^{\lfloor tN^{2\alpha}\rfloor}(\lfloor tN^{2\alpha}\rfloor-k+1)^{-\frac{1}{2}}N^{-\alpha}|x-y|^{\frac{1}{2}}+N^{-\frac{3\alpha}{2}}k^{-\frac{3}{4}} \right)^{\frac{p}{2}}.
\end{split}
\end{equation*}
It is easily seen that
\begin{equation}
\label{space}
\E |\hat{A}(\xi_{\lfloor tN^{2\alpha}\rfloor})(x)-\hat{A}(\xi_{\lfloor tN^{2\alpha}\rfloor})(y)|^p\leq C(\lambda,p,f,T)\e_{\lambda p}(x)\left( |x-y|^{\frac{p}{4}}+N^{-\frac{\alpha p}{2}} \right).
\end{equation}
For the time difference,
\begin{align*}
&~~~~\hat{A}(\xi_{\lfloor tN^{2\alpha}\rfloor})(y)-\hat{A}(\xi_{\lfloor sN^{2\alpha}\rfloor})(y) \\
&= M_{\lfloor tN^{2\alpha}\rfloor}(\psi^y_{\lfloor tN^{2\alpha}\rfloor-\cdot}) - M_{\lfloor sN^{2\alpha}\rfloor}(\psi^y_{\lfloor sN^{2\alpha}\rfloor-\cdot})\\
&=\frac{1}{N^{2\alpha}}\sum_{k=1}^{\lfloor tN^{2\alpha}\rfloor}\sum_x\sum_{z\sim x}\sum_{w=1}^{\xi_{k-1}(z)}\psi^y_{\lfloor tN^{2\alpha}\rfloor-k+1}(x)\left( \eta_k^w(z,x)-\frac{1}{2N} \right)\\
&~~~~-\frac{1}{N^{2\alpha}}\sum_{k=1}^{\lfloor sN^{2\alpha}\rfloor}\sum_x\sum_{z\sim x}\sum_{w=1}^{\xi_{k-1}(z)}\psi^y_{\lfloor sN^{2\alpha}\rfloor-k+1}(x)\left( \eta_k^w(z,x)-\frac{1}{2N} \right)\\
&=\frac{1}{N^{2\alpha}}\sum_{k=1}^{\lfloor sN^{2\alpha}\rfloor}\sum_x\sum_{z\sim x}\sum_{w=1}^{\xi_{k-1}(z)}(\psi^y_{\lfloor tN^{2\alpha}\rfloor-k+1}(x)-\psi^y_{\lfloor sN^{2\alpha}\rfloor-k+1}(x))\left( \eta_k^w(z,x)-\frac{1}{2N} \right)\\
&~~~~+\frac{1}{N^{2\alpha}}\sum_{k=\lfloor sN^{2\alpha}\rfloor+1}^{\lfloor tN^{2\alpha}\rfloor}\sum_x\sum_{z\sim x}\sum_{w=1}^{\xi_{k-1}(z)}\psi^y_{\lfloor tN^{2\alpha}\rfloor-k+1}(x)\left( \eta_k^w(z,x)-\frac{1}{2N} \right)\\
&= M^{(1)}_{\lfloor sN^{2\alpha}\rfloor} + \left( M^{(2)}_{\lfloor tN^{2\alpha}\rfloor} - M^{(2)}_{\lfloor sN^{2\alpha}\rfloor}\right),
\end{align*}

where the two martingales are
$$M^{(1)}_{\lfloor sN^{2\alpha}\rfloor} = \frac{1}{N^{2\alpha}}\sum_{k=1}^{\lfloor sN^{2\alpha}\rfloor}\sum_x\sum_{z\sim x}\sum_{w=1}^{\xi_{k-1}(z)}(\psi^y_{\lfloor tN^{2\alpha}\rfloor-k+1}(x)-\psi^y_{\lfloor sN^{2\alpha}\rfloor-k+1}(x))\left( \eta_k^w(z,x)-\frac{1}{2N} \right)$$
and 
$$M^{(2)}_{\lfloor sN^{2\alpha}\rfloor} = \frac{1}{N^{2\alpha}}\sum_{k=1}^{\lfloor sN^{2\alpha}\rfloor}\sum_x\sum_{z\sim x}\sum_{w=1}^{\xi_{k-1}(z)}\psi^y_{\lfloor tN^{2\alpha}\rfloor-k+1}(x)\left( \eta_k^w(z,x)-\frac{1}{2N} \right).$$

For $M^{(1)}_{\lfloor sN^{2\alpha}\rfloor}$, we use the similar argument as (\ref{bracket}), and get
\begin{align*}
\langle M^{(1)}\rangle_{\lfloor sN^{2\alpha}\rfloor} &\leq \sum_{k=1}^{\lfloor sN^{2\alpha}\rfloor} \frac{\| \psi^y_{\lfloor tN^{2\alpha}\rfloor-k+1}-\psi^y_{\lfloor sN^{2\alpha}\rfloor-k+1}\|_{\lambda}}{N^{2\alpha}}\left(A\xi_{k-1}e_{-\lambda},\psi^y_{\lfloor tN^{2\alpha}\rfloor-k+1}+\psi^y_{\lfloor tN^{2\alpha}\rfloor-k+1}\right)\\
&\leq \sum_{k=1}^{\lfloor sN^{2\alpha}\rfloor}\left( N^{-\frac{\alpha}{2}}(\lfloor sN^{2\alpha}\rfloor-k+1)^{-\frac{3}{4}}|t-s|^{\frac{1}{2}}+N^{-\frac{3\alpha}{2}}(\lfloor sN^{2\alpha}\rfloor-k+1)^{-\frac{3}{4}} \right)\\
&\qquad \qquad \qquad \qquad \qquad \qquad \qquad \qquad \qquad \qquad \cdot\left(A\xi_{k-1}e_{-\lambda},\psi_{\lfloor tN^{2\alpha}\rfloor-k+1}^y+\psi_{\lfloor sN^{2\alpha}\rfloor-k+1}^y\right).
\end{align*}
By BDG inequality,
\begin{align*}
\E \left| M^{(1)}_{\lfloor sN^{2\alpha}\rfloor}\right|^p &\leq \E\langle M^{(1)}\rangle^{\frac{p}{2}}_{\lfloor sN^{2\alpha}\rfloor}\\
&\leq C(p)\sum_{k=1}^{\lfloor sN^{2\alpha}\rfloor}\left( N^{-\frac{\alpha}{2}}(\lfloor sN^{2\alpha}\rfloor-k+1)^{-\frac{3}{4}}|t-s|^{\frac{1}{2}}+N^{-\frac{3\alpha}{2}}(\lfloor sN^{2\alpha}\rfloor-k+1)^{-\frac{3}{4}} \right)^{\frac{p}{2}}\\
&\qquad \qquad \qquad \qquad \qquad \qquad \qquad \qquad \qquad \qquad \cdot\E\left(A\xi_{k-1}e_{-\lambda},\psi_{\lfloor tN^{2\alpha}\rfloor-k+1}^y+\psi_{\lfloor sN^{2\alpha}\rfloor-k+1}^y\right)^{\frac{p}{2}}.
\end{align*} 
Writing $e_{-\frac{\lambda p}{2}}=e_{-\frac{3\lambda p}{2}}e_{\lambda p}$ and implementing Lemma \ref{moment} (c) give
$$\|\E(A^{p/2}(\xi_{k-1}))\|_{-\frac{3\lambda p}{2}}\leq C(\lambda,p,f,T).$$
Together with Lemma \ref{psi} (b), we get
\begin{align*}
\E\left(A\xi_{k-1}e_{-\lambda},\psi_{\lfloor tN^{2\alpha}\rfloor-k+1}^y+\psi_{\lfloor sN^{2\alpha}\rfloor-k+1}^y\right)^{\frac{p}{2}}&\leq \E\left( (A^{p/2}(\xi_{k-1}))e_{-\frac{\lambda p}{2}},\psi_{\lfloor tN^{2\alpha}\rfloor-k+1}^y+\psi_{\lfloor sN^{2\alpha}\rfloor-k+1}^y\right)\\
&\leq C(\lambda,p,T)e_{\lambda p}(y).
\end{align*}
Hence,
\begin{equation}\label{time1}
\begin{split}
\E \left| M^{(1)}_{\lfloor sN^{2\alpha}\rfloor}\right|^p&\leq C(\lambda,p,f,T)\e_{\lambda p}(y)\left( \sum_{k=1}^{\lfloor sN^{2\alpha}\rfloor}N^{-\frac{\alpha}{2}}(\lfloor sN^{2\alpha}\rfloor-k+1)^{-\frac{3}{4}}|t-s|^{\frac{1}{2}}+N^{-\frac{3\alpha}{2}}(\lfloor tN^{2\alpha}\rfloor-k+1)^{-\frac{3}{4}} \right)^{\frac{p}{2}}\\
&\leq C(\lambda,p,f,T)\e_{\lambda p}(y) \left( |t-s|^{\frac{p}{4}}+N^{-\frac{\alpha p}{2}} \right),
\end{split}
\end{equation}
where the second inequality is because of the fact that 
$$\sum_{k=1}^{\lfloor sN^{2\alpha}\rfloor}(\lfloor sN^{2\alpha}\rfloor-k+1)^{-\frac{3}{4}}\leq C(T)N^{\frac{\alpha}{2}}.$$
For $M^{(2)}_{\lfloor tN^{2\alpha}\rfloor} - M^{(2)}_{\lfloor sN^{2\alpha}\rfloor}$,
\begin{align*}
\E \left| M^{(2)}_{\lfloor tN^{2\alpha}\rfloor} - M^{(2)}_{\lfloor sN^{2\alpha}\rfloor} \right|^p\leq \E \left(\langle M^{(2)} \rangle_{\lfloor tN^{2\alpha}\rfloor}-\langle M^{(2)}\rangle_{\lfloor sN^{2\alpha}\rfloor}  \right)^{\frac{p}{2}}.
\end{align*}
Similar as the argument in (\ref{bracket}),
\begin{align*}
\langle M^{(2)} \rangle_{\lfloor tN^{2\alpha}\rfloor}-\langle M^{(2)}\rangle_{\lfloor sN^{2\alpha}\rfloor} &\leq \sum_{k=\lfloor sN^{2\alpha}\rfloor+1}^{\lfloor tN^{2\alpha}\rfloor}\frac{\| \psi^y_{\lfloor tN^{2\alpha}\rfloor-k+1} \|_{\lambda}}{N^{2\alpha}}(A\xi_{k-1}e_{-\lambda},\psi_{\lfloor tN^{2\alpha}\rfloor-k+1})\\
&\leq \sum_{k=\lfloor sN^{2\alpha}\rfloor+1}^{\lfloor tN^{2\alpha}\rfloor}\left( N^{-\alpha}(\lfloor tN^{2\alpha}\rfloor-k+1)^{-\frac{1}{2}} \right)(A\xi_{k-1}e_{-\lambda},\psi^y_{\lfloor tN^{2\alpha}\rfloor-k+1}).
\end{align*}
Thanks again to Lemma \ref{psi} (b) and Lemma \ref{moment} (c),
\begin{equation}\label{time2}
\begin{split}
\E \left| M^{(2)}_{\lfloor tN^{2\alpha}\rfloor} - M^{(2)}_{\lfloor sN^{2\alpha}\rfloor} \right|^p &\leq C(\lambda,p,f,T)\e_{\lambda p}(y)\left(  \sum_{k=\lfloor sN^{2\alpha}\rfloor+1}^{\lfloor tN^{2\alpha}\rfloor} N^{-\alpha}(\lfloor tN^{2\alpha}\rfloor-k+1)^{-\frac{1}{2}} \right)^{\frac{p}{2}}\\
&\leq C(\lambda,p,f,T)\e_{\lambda p}(y)|t-s|^{p/4}.
\end{split}
\end{equation}
Summarising (\ref{space})(\ref{time1})(\ref{time2}), we can get (\ref{eqn: tightness}).
\end{proof}
Tightness of $\left\{A(\xi_{\lfloor tN^{2\alpha}\rfloor}),N\geq 1\right\}$ follows from Lemma \ref{tight}.

\begin{lemma}\label{measure density}
For $\phi:\mathbb{Z}/N^{1+\alpha}\rightarrow [0,\infty)$ and $\lambda>0$,
\begin{align*}
|(\nu_k^N,\phi)-(A(\xi_k),\phi)|\leq \|D(\phi,N^{-\alpha})\|_{\lambda}(\nu_k^N,\e_{-\lambda}),
\end{align*}
with $D(\phi,N^{-\alpha})$ as defined in (\ref{eqn:amp}).
\end{lemma}
\begin{proof}
\begin{align*}
(A(\xi_k),\phi) &= \frac{1}{N^{1+\alpha}}\sum_x A(\xi_k)(x)\phi(x)\\
&=\frac{1}{2N^{1+2\alpha}}\sum_x\sum_{y\sim x}\xi_k(y)\phi(x)\\
&=\frac{1}{2N^{1+2\alpha}}\sum_x\sum_{y\sim x}\xi_k(y)(\phi(x)-\phi(y))+(\nu_k^N,\phi).
\end{align*}
Therefore,
\begin{align*}
|(\nu_k^N,\phi)-(A(\xi_k),\phi)| &\leq \frac{1}{2N^{1+2\alpha}}\sum_x\sum_{y\sim x}\xi_k(y)D(\phi,N^{-\alpha})(y)\\
&=(\nu_k^N,D(\phi,N^{-\alpha}))\\
&\leq \|D(\phi,N^{-\alpha})\|_{\lambda}(\nu_k^N,\e_{-\lambda}).
\end{align*}
\end{proof}
Lemma \ref{measure density} together with (a) of Lemma \ref{moment} give that
\begin{align*}
\E\left(\sup_{1\leq k\leq \lfloor tN^{2\alpha}\rfloor}\|(\nu_k^N,\phi)-(A(\xi_k),\phi)\|^p \right)\leq C(\lambda,p,f,T)\|D(\phi,N^{-\alpha})\|_{\lambda}^p.
\end{align*}
This implies the tightness of $\left\{(\nu^N_{\lfloor tN^{2\alpha}\rfloor},\phi),N\geq 1\right\}$ for each such test function and therefore the tightness of $\left\{\nu^N_{\lfloor tN^{2\alpha}\rfloor},N\geq 1\right\}$ as a measure-valued process under vague topology. Hence, with probability one, for all $T>0,\lambda>0$, and test function $\phi\in C^2_0(\R)$ with compact support, we can have a subsequence of $A(\xi_{\lfloor tN^{2\alpha}\rfloor}^N)$ and $\nu_{\lfloor tN^{2\alpha}\rfloor}^N$ such that
\begin{align*}
\sup_{t\leq T}\| A(\xi_{\lfloor tN^{2\alpha}\rfloor}^N)-u_t\|_{-\lambda}\rightarrow 0, \text{ as }N \rightarrow \infty,\\
\sup_{t\leq T}\left| \int \phi(x)\nu_{\lfloor tN^{2\alpha}\rfloor}^N(dx)-\int \phi(x)\nu_t(dx) \right|\rightarrow 0, \text{ as }N\rightarrow \infty.
\end{align*}
From Lemma \ref{measure density}, we can see that $\nu_t$ is absolutely continuous with density $\nu_t(dx)=u_t(x)dx$. By substituting $\phi^N_k=\phi\in C^2_0(\R)$ as the test function with compact support in the decomposition (\ref{decomp}),  we can see that
\begin{equation*}
M^N_{\lfloor tN^{2\alpha}\rfloor}(\phi) = (\nu^N_{\lfloor tN^{2\alpha}\rfloor},\phi)-(A(\xi_0),\phi)-\sum_{k=1}^{\lfloor tN^{2\alpha}\rfloor}\frac{1}{N^{2\alpha}}(\nu^N_k,\Delta_D \phi)
\end{equation*}
is a martingale and every term on the right-hand side converges almost surely by Lemma \ref{tight}. Hence $M^N_{\lfloor tN^{2\alpha}\rfloor}(\phi)$ converges to a local martingale
\begin{equation}
\label{decomp cont}
\begin{split}
m_t(\phi)&=\int \phi(x)\nu_t(dx)-\int \phi(x)\nu_0(dx)-\int_0^t \int \frac{1}{6}\Delta \phi(x)\nu_s(dx)ds\\
&= \int \phi(x)u_t(x)dx-\int \phi(x)f(x)dx-\int_0^t \int \frac{1}{6}\Delta \phi(x)u_s(x)dxds,
\end{split}
\end{equation}
which is continuous since every term on the right-hand side is continuous. Moreover, from (\ref{bracket}),
\begin{align*}
(M_{\lfloor tN^{2\alpha}\rfloor}^N)^2-\sum_{k=1}^{\lfloor tN^{2\alpha}\rfloor}\frac{1}{N^{2\alpha}}(A(\xi_{k-1}),\phi^2)\left(1-\frac{1}{2N}\right)
\end{align*}
is a martingale. As $N\rightarrow \infty$,
\begin{equation}\label{quad mart}
m_t^2(\phi)-\int_0^t \int \phi^2(x)u_s(x)dx ds
\end{equation}
is also a continuous local martingale. (\ref{decomp cont}) and (\ref{quad mart}) prove that any subsequential weak limit $\nu_t(dx)=u_t(x)dx$ solves (\ref{eqn: brw spde}). The uniqueness follows from Theorem 5.7.1 of \cite{dawson} which finishes the proof of Theorem \ref{thm: brw}.

\subsection{Multiple particles at one site}
\label{sec:multiple}
In this subsection, we show the probability of multiple particles at one site is negligible in the branching envelope. Then, the state function can be reduced from its number of particles to that it is occupied or vacant. 
We will first show a property (see Lemma \ref{finite Tk}) that refers to the weak limit of the envelope process. This is then used to deal with the discrete process.

Let $X_t$ denote the total mass of this system, that is
\begin{align*}
X_t &= (\nu_t,1)=\int u_t(x)dx.
\end{align*}
We then have that
\begin{align*}
X_t = \int_0^t\int \sqrt{u_s(x)} W(ds,dx),
\end{align*}
and therefore its quadratic variation is $\langle X\rangle_t=\int_0^t \int u_s(x)dxds$. Hence
\begin{align*}
\begin{cases}
X_t &= \int_0^t \sqrt{X_s}dB_s,\\
X_0&= 2.
\end{cases}
\end{align*}
For $k\geq 1$, let $T_k$ denote the stopping time given by
\begin{align*}
T_k = \inf \left\{ t>0: X_t\geq 2^k  \text{ or }\int_0^t X_sds\geq 2^{2k} \right\}=T'_k\wedge T''_k,
\end{align*}
where
\begin{align*}
T'_k=\inf \left\{ t>0:X_t\geq 2^k  \right\},~T''_k=\inf \left\{ \int_0^t X_sds\geq 2^{2k} \right\}.
\end{align*}

\begin{lemma}\label{finite Tk}
For a fixed initial condition $f$ which is continuous and compact supported satisfying $f(x)=1$ for $x\in[-1,1]$, there exists constant $C$ so that
\begin{align*}
\P(T_{k}<\infty)\leq C2^{-k}, k\in \mathbb{N}.
\end{align*}
\end{lemma}
\begin{proof}
$\P(T_k<\infty)$ can be decomposed as
\begin{equation}\label{Tk}
\begin{split}
\P(T_k<\infty) &= \P\left(T'_k<\infty, T'_k<T''_k\right)+\P\left(T''_k<\infty,T''_k<T'_k\right)\\
&\leq \P\left(T'_k<\infty\right)+\P\left(T''_k<T'_k\right).
\end{split}
\end{equation}
If we denote $H_0=\inf\{t>0:X_t=0\}$ as the first hitting time of zero then $\P(T'_k<\infty)=\P(T'_k<H_0)$, hence the first term on the RHS of (\ref{Tk}) is simply
\begin{align*}
\P\left(T'_k<\infty\right)<\frac{X_0}{2^k}.
\end{align*}
The event $\{ T''_k<T'_k\}\subset\cup_{j=1}^{k-1}A_j$, where
\begin{align*}
A_j=\left\{ T'_j<\infty, \int_{T'_j}^{T'_{j+1}}X_udu\geq \frac{6\cdot2^{2k}}{\pi (k-j)^2} \right\}.
\end{align*}
By using the Markov property of $X_t$,
\begin{align}\label{Aj}
\P(A_j)\leq \P(T'_j<\infty)\P\left( \int_{T'_j}^{T'_{j+1}}X_udu\geq \frac{6\cdot2^{2k}}{\pi (k-j)^2} \right).
\end{align}
The total mass satisfies
\begin{align*}
X_t-X_s=\int_s^t \sqrt{X_u}dB_u,
\end{align*}
then
\begin{align*}
\E \left( (X_t-X_s)^2|X_s\right) = \E \left( \int_s^t X_udu \middle |  X_s\right).
\end{align*}
Hence we can get
\begin{align*}
\E \int_{T'_j}^{T'_{j+1}}X_udu\leq 2^{2j}.
\end{align*}
From this, the second term in (\ref{Aj}) can be bounded by using the Markov inequality
\begin{align*}
\P\left( \int_{T'_j}^{T'_{j+1}}X_udu\geq \frac{6\cdot2^{2k}}{\pi (k-j)^2} \right) \leq \frac{C(k-j)^2}{2^{2(k-j)}}.
\end{align*}
Therefore,
\begin{align*}
\P(A_j)\leq \frac{1}{2^k}\cdot \frac{CX_0(k-j)^2}{2^{k-j}}.
\end{align*}
After plugging in $\P(T_k<\infty)\leq \P(T'_k<\infty)+\sum_{j=1}^{k-1}\P(A_j)$, we have that,
\begin{align*}
\P(T_k<\infty)\leq \frac{C}{2^k}.
\end{align*}
\end{proof}
\begin{corollary}
\label{total mass}
For any $t>0$, $$\int u_t(x)dx\text{ and }\int_0^t\int u_s(x)dxds$$ are finite with probability one.
\end{corollary}
Then for the discrete state function, we have:
\begin{lemma}
\label{cor: multiple}
For any $k\in \mathbb{N}$ and any $x\in \Z/N^{1+\alpha}$, 
$$\P(\xi_k(x)>1\mid \F_{k-1})<(A\xi_{k-1}(x))^2N^{2\alpha-2},$$
where $\{\F_k\}_{k\geq0}$ is the natural filtration $\F_k=\sigma(\{\xi_j, 0\leq j\leq k\})$.
\end{lemma}
\begin{proof}
In the discrete system, we have
$$\xi_k(x)=\sum_{y\sim x}\sum_{w=1}^{\xi_{k-1}(y)}\eta^w_k(y,x).$$
Denote $N'=\sum_{y\sim x}\xi_{k-1}(y)=2N^{\alpha}(A\xi_{k-1}(x))$. Given $\{\xi_{k-1}(y),y\in\Z/N^{1+\alpha}\}$, we have
$$\xi_k(x)\overset{d}{=}\text{Binomial}(N',1/(2N)).$$
Note $\P_{k-1}(\cdot)$ as conditional probability given $\F_{k-1}$.
\begin{align*}
\P_{k-1}(\xi_k(x)\geq 2) &=1-\P_{k-1}(\xi_k(x)=0)-\P_{k-1}(\xi_k(x)=1)\\
&=1-\left(1-\frac{1}{2N}\right)^{N'}-N'\frac{1}{2N}\left(1-\frac{1}{2N}\right)^{N'-1}\\
&\leq \left( \frac{N'}{2N} \right)^2\\
&=N^{2\alpha-2}(A\xi_{k-1}(x))^2.
\end{align*}
\end{proof}
Since the branching envelope dominates the true horizontal process, this property will also hold for the true horizontal process.

\section{The true horizontal process}
\label{sec:real}
The true process we consider is dominated by the branching random walk in the preceding section, which means that at each time step, the particles will move and reproduce following the mechanism of $\xi_k$. But if the site $x$ has been occupied by some particle before, then it cannot be occupied again. We denote $\{\hat{\xi}_{k}(x)\in \{0,1\},k\in \mathbb{Z}_+, x\in \mathbb{Z}/N^{1+\alpha}\}$ as the mechanism of the true process. 
It can be expressed as
\begin{align*}
\hat{\xi}_{k+1}(x)=\begin{cases}
1 &\text{ if } \sum_{j\leq k}\hat{\xi}_j(x)=0 \text{ and }\sum_{y\in\mathcal{N}_k(x)}\eta_{k+1}(y,x)\geq 1,\\
0 &\text{otherwise},
\end{cases}
\end{align*}
where $\mathcal{N}_k(x)=\{y\sim x:\hat\xi_k(y)=1\}$ having cardinality $N_k(x)=\sum_{y\sim x}\hat{\xi}_k(y)$ and $(\eta_{k+1}(y,x))_{k,y,x}$ is an i.i.d. sequence with distribution $\text{Bernoulli}(1/(2N))$.
$\hat{\xi}_{k+1}(x)$ can be rewritten as
\begin{equation}
\label{real mech}
\begin{split}
\hat{\xi}_{k+1}(x)&=\indic_{\{\sum_{y\in\mathcal{N}_k(x)}\eta_{k+1}(y,x)\geq 1\}}\left( 1-\sum_{j\leq k}\hat{\xi}_j(x) \right)\\
&=\underbrace{\sum_{y\in\mathcal{N}_k(x)}\eta_{k+1}(y,x)\left( 1-\sum_{j\leq k}\hat{\xi}_j(x) \right)}_{\text{main}}\\
&~~~~+\underbrace{\left(\indic_{\{\sum_{y\in\mathcal{N}_k(x)}\eta_{k+1}(y,x)\geq 1\}}-\sum_{y\in\mathcal{N}_k(x)}\eta_{k+1}(y,x)\right)\left( 1-\sum_{j\leq k}\hat{\xi}_j(x) \right)}_{\text{error}}\\
&=\left( \sum_{y\in\mathcal{N}_k(x)}\eta_{k+1}(y,x) \right)\left( 1-\sum_{j\leq k}\hat{\xi}_j(x) \right)+E_k(x).
\end{split}
\end{equation}

The main goal of this section is to describe the limiting behaviour of the true horizontal process, summarized in the following result:
\begin{theorem}
\label{thm real spde}
Assume that as $N\rightarrow\infty$, $A(\hat\xi_0)$ converges in $\mathcal{C}$ to a continuous function $f$ with compact support.
When $\alpha=1/5$, as $N\rightarrow \infty$, $A(\hat{\xi}_{\lfloor tN^{2\alpha}\rfloor})(x)$ converges in law to $\hat{u}_t(x)$, which is the unique in law solution to the following SPDE
\begin{align}\label{real spde}
\begin{cases}
\frac{\partial \hat{u}_t}{\partial t}=\frac{1}{6}\Delta \hat{u}_t-\hat{u}_t\int_0^t\hat{u}_sds+\sqrt{\hat{u}_t}\dot{W}(t,\cdot) \\
\hat{u}_0=f,
\end{cases}
\end{align}
where $\dot{W}$ is the space-time white noise.
\end{theorem}

\begin{remark}
In the later proof, we frequently choose $f$ such that $f(x)=1$ for $x\in[-r,r],r>0$.
\end{remark}

The proof is given in the next two subsections: we first prove the tightness and that any weak limit satisfies (\ref{real spde}), and in Subsection \ref{sec: uniqueness} we prove the uniqueness. 
\subsection{Limit behaviour of the rescaled horizontal process}
Similarly to (\ref{brw mech}) and by (\ref{real mech}),
\begin{align*}
\hat{\xi}_{k+1}(x) &=\left( \frac{1}{2N}\sum_{y\sim x}\hat{\xi}_k(y)+\sum_{y\in\mathcal{N}_k(x)}\left(\eta_{k+1}(y,x)-\frac{1}{2N} \right) \right)\left( 1-\sum_{j\leq k}\hat{\xi}_j(x) \right)+E_k(x)\\
&=\hat{\xi}_k(x)+\frac{1}{2N}\sum_{y\sim x}(\hat{\xi}_k(y)-\hat{\xi}_k(x))+\sum_{y\in\mathcal{N}_k(x)}\left(\eta_{k+1}(y,x)-\frac{1}{2N} \right)\\
&~~~~-\frac{1}{2N}\sum_{y\sim x}\sum_{j\leq k}\hat{\xi}_k(y)\hat{\xi}_j(x)-\sum_{y\in\mathcal{N}_k(x)}\left(\eta_{k+1}(y,x)-\frac{1}{2N} \right)\sum_{j\leq k}\hat{\xi}_j(x)+E_k(x)\\
&=\hat{\xi}_k(x)+\frac{1}{2N}\sum_{y\sim x}(\hat{\xi}_k(y)-\hat{\xi}_k(x))+\sum_{y\in\mathcal{N}_k(x)}\left(\eta_{k+1}(y,x)-\frac{1}{2N} \right)\\
&~~~~-\frac{1}{N^{1-\alpha}}A(\hat{\xi}_k(x))\sum_{j\leq k}\hat{\xi}_j(x)-\sum_{y\in\mathcal{N}_k(x)}\left(\eta_{k+1}(y,x)-\frac{1}{2N} \right)\sum_{j\leq k}\hat{\xi}_j(x)+E_k(x).
\end{align*}

Denote $\hat{\nu}^N_k=\frac{1}{N^{2\alpha}}\sum_x\hat{\xi}_k(x)\delta_x$ as the measure generated by $\hat{\xi}_k$. Choose test function $\phi^N$ satisfying (\ref{test fun cond}) and sum by parts,
\begin{align*}
(\hat{\nu}^N_k,\phi^N_k)-(\hat{\nu}^N_{k-1},\phi^N_{k-1}) &=(\hat{\nu}^N_k,\phi^N_k-\phi^N_{k-1})+(\hat{\nu}^N_{k-1},N^{-2\alpha}\Delta_D \phi^N_{k-1})+d_k^{(1)}(\phi^N)\\
&~~~~-\frac{1}{N^{1-\alpha}}\sum_{j\leq k-1}(A(\hat{\xi}_{k-1}) \phi^N_{k-1},\hat{\nu}^N_j)-d_k^{(2)}(\phi)+E_k(\phi^N),
\end{align*}
with the error term$$E_k(\phi^N)=\frac{1}{N^{2\alpha}}\sum_xE_k(x)\phi^N_k(x),$$
and martingale terms
\begin{align*}
d_k^{(1)}(\phi^N)&=\frac{1}{N^{2\alpha}}\sum_x\phi^N_{k-1}(x)\sum_{y\in\mathcal{N}_{k-1}(x)}\left(\eta_{k}(y,x)-\frac{1}{2N}\right),\\
d_k^{(2)}(\phi^N)&=\frac{1}{N^{2\alpha}}\sum_x\phi^N_{k-1}(x)\sum_{y\in\mathcal{N}_{k-1}(x)}\left(\eta_{k}(y,x)-\frac{1}{2N}\right)\sum_{j\leq k-1}\hat{\xi}_j(x).
\end{align*}
Summing $k$ from $1$ to $n$, we can get a semimartingale decomposition
\begin{equation}
\label{real proc decomp}
\begin{split}
(\hat{\nu}^N_n,\phi^N_n)-(A(\hat{\xi}_0),\phi_0)&=(\hat{\nu}^N_n,\phi^N_n-\phi^N_{n-1})+\sum_{k=1}^{n-1}(\hat{\nu}^N_k,\phi^N_k-\phi^N_{k-1}+N^{-2\alpha}\Delta_D \phi^N)\\
&~~~~-\sum_{k=1}^{n-1}\sum_{j\leq k}\frac{1}{N^{1-\alpha}}(A(\hat{\xi}_k)\phi^N_k,\hat{\nu}^N_j)+\hat{M}_n(\phi^N)+\sum_{k=1}^nE_k(\phi^N),
\end{split}
\end{equation}
where the martingale
$\hat{M}_n(\phi^N)=M_n^{(1)}(\phi^N)-M_n^{(2)}(\phi^N)=\sum_{k=1}^n\left(d_k^{(1)}(\phi^N)-d_k^{(2)}(\phi^N)\right)$ has square variation
\begin{equation}
\label{real proc bracket}
\begin{split}
\langle \hat{M}(\phi^N)\rangle_n &=\sum_{k=1}^n \E_{k-1}(d_k^{(1)}(\phi^N)-d_k^{(2)}(\phi^N))^2\\
&=\sum_{k=1}^n\frac{1}{2N^{1+4\alpha}}\sum_x\sum_{y\sim x}\hat{\xi}_{k-1}(y)\left( 1-\left(\sum_{j\leq k-1}\hat{\xi}_j(x)\right)^2 \right)(\phi^N_{k-1})^2(x)\left(1-\frac{1}{2N}\right)\\
&\leq \sum_{k=1}^n\frac{\|\phi^N_{k-1}\|_0}{2N^{1+4\alpha}}\sum_x\sum_{y\sim x}\hat{\xi}_{k-1}(y)\left( 1-\sum_{j\leq k-1}\hat{\xi}_j(x)\right)\phi^N_{k-1}(x)\\
&=\sum_{k=1}^n\frac{\|\phi^N_{k-1}\|_0}{N^{2\alpha}}(A(\hat{\xi}_{k-1}),\phi)-\frac{\|\phi\|_0}{N^{1+\alpha}}\sum_{k=1}^n\sum_{j\leq k-1}(A(\hat{\xi}_{k-1})\phi^N_{k-1},\hat{\nu}^N_j),
\end{split}
\end{equation}
where we use the fact that $\sum_{j\leq k}\hat{\xi}_j(x)\in \{0,1\}$ to get the first inequality.
 We first show that the error term in (\ref{real proc decomp}) is negligible.
\begin{lemma}
\label{small error}
When $\alpha<1/3$, for $t\leq T$, the cumulative error term over time $\lfloor tN^{2\alpha}\rfloor$
$$\frac{1}{N^{2\alpha}}\sum_{k=1}^{\lfloor tN^{2\alpha}\rfloor}\sum_xE_k(x)\phi^N_k(x)\rightarrow 0\text{ in }L^2 \text{ as }N\rightarrow\infty.$$
The test function $\phi^N_k(x)$ is chosen as the discrete approximation of $\phi(t,x):\R_+\times\R\rightarrow\R$ by taking $\phi^N(k,x)=\phi\left(\frac{k}{N^{2\alpha}},x\right)$ for $x\in\Z/N^{1+\alpha}$, where $\phi$ is compact supported and twice differentiable in $t$ and $x$.
\end{lemma}
\begin{proof}
By H\"older inequality,
\begin{align*}
&~~~~\E\left| \frac{1}{N^{2\alpha}}\sum_{k=1}^{\lfloor tN^{2\alpha}\rfloor}\sum_xE_k(x)\phi^N_k(x) \right|^2\\
&\leq \frac{t}{N^{2\alpha}}\sum_{k=1}^{\lfloor tN^{2\alpha}\rfloor}\E\left(\sum_xE_k(x)\phi^N_k(x)\right)^2\\
&\leq \frac{T}{N^{2\alpha}}\sum_{k=1}^{\lfloor tN^{2\alpha}\rfloor}\sum_x\E\left(\indic_{\{\sum_{y\in\mathcal{N}_k(x)}\eta_{k+1}(y,x)\geq 1\}}-\sum_{y\in\mathcal{N}_k(x)}\eta_{k+1}(y,x)\right)^2(\phi^N_k)^2(x),
\end{align*}
where in the second inequality, we used the facts that $\left|1-\sum_{j\leq k}\hat{\xi}_j(x)\right|\leq 1$ and, given $\F_k$, $$\indic_{\{\sum_{y\in\mathcal{N}_k(x)}\eta_{k+1}(y,x)\geq 1\}}-\sum_{y\in\mathcal{N}_k(x)}\eta_{k+1}(y,x),x\in \Z/N^{1+\alpha}$$ are conditionally independent. Following similar reason as Lemma \ref{cor: multiple},
$$\E\left(\indic_{\{\sum_{y\in\mathcal{N}_k(x)}\eta_{k+1}(y,x)\geq 1\}}-\sum_{y\in\mathcal{N}_k(x)}\eta_{k+1}(y,x)\right)^2\leq\frac{\E [N_k(x)^2]}{4N^2}.$$
$N_k(x)=\sum_{y\sim x}\hat{\xi}_k(y)$ can be written as $2N^{\alpha}A\hat{\xi}_k(x)$. Hence
\begin{align*}
\E\left| \frac{1}{N^{2\alpha}}\sum_{k=1}^{\lfloor tN^{2\alpha}\rfloor}\sum_xE_k(x)\phi^N_k(x) \right|^2 &\leq \frac{T}{N^2}\sum_{k=1}^{\lfloor tN^{2\alpha}\rfloor}\sum_x \E(A\hat{\xi}_k(x))^2(\phi^N_k)^2(x)\\
&\leq \frac{C(\lambda,f,T)}{N^{1-3\alpha}}\sum_{k=1}^{\lfloor tN^{2\alpha}\rfloor}\frac{1}{N^{2\alpha}}(\phi^N_k,e_{\lambda})^2 \text{ (Lemma \ref{moment} (c))}.
\end{align*}
The result follows by using the properties of test functions (\ref{test fun cond}).
\end{proof}

Choosing $\phi^N_k=\psi_{n-k}$ as in Section \ref{green fun}, we can obtain
\begin{align}
\label{real proc approx}
A(\hat{\xi}_n)(x)=(\hat{\nu}^N_0,\psi_n^x)-\sum_{k=1}^{n}\sum_{j\leq k-1}\frac{1}{N^{1-\alpha}}\left(A(\hat{\xi}_{k-1})\psi_{n-k}^x,\hat{\nu}^N_j\right)+\hat{M}_n(\psi_{n-\cdot}^x)+\sum_{k=1}^nE_k(\psi^x_{n-k}).
\end{align}

Since $\hat{\xi}_k(x)$ is dominated by $\xi_k(x)$, the estimations in Lemma \ref{moment} also hold for $\hat{\xi}_k(x)$. As in Section \ref{tightness}, we will use the estimations in Lemma \ref{psi} and Lemma \ref{moment} to get the tightness of $A(\hat{\xi}_{\lfloor tN^{2\alpha}\rfloor})(x)$. We assume that  the linear interpolation of $A(\hat{\xi}_0)$ converges to $f$ under $\|\cdot\|_{-\lambda}$ for any $\lambda>0$ as $N\rightarrow \infty$ and let
$$\hat{A}(\hat{\xi}_k)(x)=A(\hat{\xi}_k)(x)-(\hat{\nu}^N_0,\psi_k^x).$$

\begin{lemma}
\label{real proc tight}
When $\alpha=1/5$, for $0\leq s\leq t\leq T$, $x,y\in \mathbb{Z}/N^{1+\alpha}$, $|t-s|\leq 1$, $|x-y|\leq 1$, $\lambda>0$ and $p\geq 2$,
\begin{equation}
\label{eqn: true tightness}
\E |\hat{A}(\hat{\xi}_{\lfloor tN^{2\alpha}\rfloor})(x)-\hat{A}(\hat{\xi}_{\lfloor sN^{2\alpha}\rfloor})(y)|^p \leq C(\lambda,p,f,T)\e_{\lambda p}(x)\left( |x-y|^{\frac{p}{4}}+|t-s|^{\frac{p}{4}}+N^{-\frac{\alpha p}{2}} \right).
\end{equation}
\end{lemma}

\begin{proof}
We first deal with the error term and the remaining terms will be shown as in the proof of Lemma \ref{tight}, where we decompose this difference into space and time differences.

The error term is $$\frac{1}{N^{2\alpha}}\sum_{k=1}^{\lfloor tN^{2\alpha}\rfloor}\sum_{x'}\psi^z_{\lfloor tN^{2\alpha}\rfloor-k}(x')E_k(x'),\text{ for }z=x \text{ or }y,$$ where $$E_k(x')=\indic_{\{\sum_{y\in\mathcal{N}_k(x')}\eta_{k+1}(y,x)\geq 1\}}-\sum_{y\in\mathcal{N}_k(x')}\eta_{k+1}(y,x).$$
We can decompose $E_k(x')=E_k^{(1)}(x')+E_k^{(2)}(x')$, where
\begin{align*}
E_k^{(1)}(x')&=\E[E_k(x')\mid \F_k],
\end{align*}
satisfying
$$
\left|E_k^{(1)}(x)\right|\leq\left|1-\left(1-\frac{1}{2N}\right)^{N_k(x')}-\frac{N_k(x')}{2N}\right|\leq\frac{N_k(x')^2}{4N^2},
$$
and
\begin{align*}
E_k^{(2)}(x)&=E_k(x')-\E[E_k(x')\mid \F_k].
\end{align*}

With respect to the first term $E_k^{(1)}$, we have
\begin{align*}
&~~~~\E\left|\frac{1}{N^{2\alpha}}\sum_{k=1}^{\lfloor tN^{2\alpha}\rfloor}\sum_{x'}\psi^z_{\lfloor tN^{2\alpha}\rfloor-k}(x')E_k^{(1)}(x')\right|^p\\
&\leq\frac{C(p,T)}{N^{2\alpha}}\sum_{k=1}^{\lfloor tN^{2\alpha}\rfloor}\E\left(\sum_{x'}\psi^z_{\lfloor tN^{2\alpha}\rfloor-k}(x')E_k^{(1)}(x')\right)^p\\
&\leq\frac{C(p,T)}{N^{2\alpha+(2-2\alpha)p}}\sum_{k=1}^{\lfloor tN^{2\alpha}\rfloor}\E\left(\sum_{x'}(A\hat{\xi}_k(x'))^2\psi^z_{\lfloor tN^{2\alpha}\rfloor-k}(x')\right)^p\\
&\leq\frac{C(p,T)}{N^{2\alpha+(2-2\alpha)p}}\sum_{k=1}^{\lfloor tN^{2\alpha}\rfloor}\E\left(\sum_{x}(A\hat{\xi}_k(x'))^{2p}\psi_{\lfloor tN^{2\alpha}\rfloor-k}(x')\right)\cdot\left(\sum_{x'}\psi^z_{\lfloor tN^{2\alpha}\rfloor-k}(x')\right)^{p-1}\\
&\leq\frac{C(p,T)N^{(1+\alpha)(p-1)}}{N^{2\alpha+(2-2\alpha)p}}\sum_{k=1}^{\lfloor tN^{2\alpha}\rfloor}\E\left(\sum_{x'}(A\hat{\xi}_k(x'))^{2p}\psi^z_{\lfloor tN^{2\alpha}\rfloor-k}(x')\right) \text{ (Lemma \ref{psi} (a))}\\
&\leq\frac{C(\lambda,p,f,T)N^{(1+\alpha)(p-1)}}{N^{2\alpha+(2-2\alpha)p}}\sum_{k=1}^{\lfloor tN^{2\alpha}\rfloor}\left(\sum_{x'}\psi^z_{\lfloor tN^{2\alpha}\rfloor-k}(x')e_{\lambda p}(x')\right) \text{ (Lemma \ref{moment} (c))}\\
&\leq C(\lambda,p,f,T)e_{\lambda p}(z)N^{-(1-3\alpha)p} \text{ (Lemma \ref{psi} (b))}.
\end{align*}
Moreover,
$$M_n^{(2)}=\frac{1}{N^{2\alpha}}\sum_{k=1}^{n}\sum_{x'}\psi_{\lfloor tN^{2\alpha}\rfloor-k}(x')E_k^{(2)}(x'),n\leq \lfloor tN^{2\alpha}\rfloor.$$
is a martingale. Hence,
\begin{align*}
\langle M^{(2)}\rangle_{\lfloor tN^{2\alpha}\rfloor}\leq \frac{C}{N^{2+2\alpha}}\sum_{k=1}^{\lfloor tN^{2\alpha}\rfloor}\sum_{x'}(A\hat{\xi}_k(x'))^2(\psi^z_{\lfloor tN^{2\alpha}\rfloor-k}(x'))^2.
\end{align*}
By BDG inequality, we have
\begin{align*}
&~~~~\E\left|\frac{1}{N^{2\alpha}}\sum_{k=1}^{\lfloor tN^{2\alpha}\rfloor}\sum_{x'}\psi^z_{\lfloor tN^{2\alpha}\rfloor-k}(x')E_k^{(2)}(x')\right|^p\\
&\leq\frac{C(p,T)}{N^{(1+\alpha)p-2\alpha(p/2-1)}}\sum_{k=1}^{\lfloor tN^{2\alpha}\rfloor}\E\left(\sum_{x'}(A\hat{\xi}_k(x'))^2(\psi^z_{\lfloor tN^{2\alpha}\rfloor-k}(x'))^2\right)^{\frac{p}{2}}\\
&\leq\frac{C(\lambda,p,T)e_{\lambda p}(z)}{N^{p+2\alpha}}\sum_{k=1}^{\lfloor tN^{2\alpha}\rfloor}N^{\frac{\alpha p}{2}}(\lfloor tN^{2\alpha}\rfloor-k)^{-\frac{p}{4}}
\E\left(\sum_{x'}(A\hat{\xi}_k(x'))^2\psi^z_{\lfloor tN^{2\alpha}\rfloor-k} (x')e_{-2\lambda}(x')\right)^{\frac{p}{2}}\text{ (Lemma \ref{psi} (c))}\\
&\leq\frac{C(\lambda,p,T)e_{\lambda p}(z)}{N^{p+2\alpha}}\sum_{k=1}^{\lfloor tN^{2\alpha}\rfloor}N^{\frac{\alpha p}{2}}(\lfloor tN^{2\alpha}\rfloor-k)^{-\frac{p}{4}}\E\left(\sum_{x'}(A\hat{\xi}_k(x'))^pe_{-\lambda p}(x')\psi^z_{\lfloor tN^{2\alpha}\rfloor-k}(x')\right)\left(\sum_{x'}\psi^z_{\lfloor tN^{2\alpha}\rfloor-k}(x')\right)^{\frac{p}{2}-1}\\
&\leq\frac{C(\lambda,p,f,T)e_{\lambda p}(z)N^{(1+\alpha)\frac{p}{2}}}{N^{p+2\alpha}}\sum_{k=1}^{\lfloor tN^{2\alpha}\rfloor}N^{\frac{\alpha p}{2}}(\lfloor tN^{2\alpha}\rfloor-k)^{-\frac{p}{4}} \text{ (Lemma \ref{psi} (a), (b) and Lemma \ref{moment} (c))}\\
&\leq C(\lambda,p,f,T)e_{\lambda p}(z)N^{-\frac{1-\alpha}{2}p},
\end{align*}
where the third inequality is from the fact that $(\psi^z_k,1)=1$ and Lemma \ref{moment} (c).

To get the estimation of space difference, first we need to deal with
\begin{align*}
&~~~~\E \left| \sum_{k=1}^{\lfloor tN^{2\alpha}\rfloor}\sum_{j\leq k-1}\frac{1}{N^{1-\alpha}}\left(A(\hat{\xi}_{k-1})(\psi^x_{\lfloor tN^{2\alpha}\rfloor-k}-\psi^y_{\lfloor tN^{2\alpha}\rfloor-k}),\hat{\nu}^N_j\right) \right|^p\\
&\leq C(\lambda,p,f,T)\e_{\lambda p}(x)\left( \frac{1}{N^{1-\alpha}}\sum_{k=1}^{\lfloor tN^{2\alpha}\rfloor}\sum_{j\leq k-1}(\lfloor tN^{2\alpha}\rfloor-k+1)^{-\frac{1}{2}}N^{\alpha}|x-y|^{\frac{1}{2}}+N^{\frac{\alpha}{2}}(\lfloor tN^{2\alpha}\rfloor-k+1)^{-\frac{3}{4}} \right)^p\\
&\leq C(\lambda,p,f,T)\e_{\lambda p}(x)\left( N^{(5\alpha-1)p}|x-y|^{\frac{p}{2}}+N^{(4\alpha-1)p} \right)\\
&\leq C(\lambda,p,f,T)\e_{\lambda p}(x)\left( |x-y|^{\frac{p}{2}}+N^{-\alpha p} \right),
\end{align*}
where the last inequality is because of $\alpha=1/5$. Next, we will use BDG inequality to estimate
\begin{align*}
\E\left| M_{\lfloor tN^{2\alpha}\rfloor}^{(2)}(\psi_{\lfloor tN^{2\alpha}\rfloor-\cdot}^x-\psi_{\lfloor tN^{2\alpha}\rfloor\cdot}^y) \right|^p &\leq \E \langle M^{(2)}(\psi_{\lfloor tN^{2\alpha}\rfloor-\cdot}^x-\psi_{\lfloor tN^{2\alpha}\rfloor\cdot}^y) \rangle_{\lfloor tN^{2\alpha}\rfloor}^{\frac{p}{2}}.
\end{align*}
As the argument in (\ref{real proc bracket}),
\begin{align*}
&~~~~ \langle M^{(2)}(\psi_{\lfloor tN^{2\alpha}\rfloor-\cdot}^x-\psi_{\lfloor tN^{2\alpha}\rfloor-\cdot}^y) \rangle_{\lfloor tN^{2\alpha}\rfloor}\\
&\leq \frac{1}{N^{1+\alpha}}\sum_{k=1}^{\lfloor tN^{2\alpha}\rfloor}\sum_{j\leq k-1}\|\psi_{\lfloor tN^{2\alpha}\rfloor-k+1}^x-\psi_{\lfloor tN^{2\alpha}\rfloor-k+1}^y\|_{\lambda}(A(\hat{\xi}_{k-1})e_{-\lambda}(\psi_{\lfloor tN^{2\alpha}\rfloor-k+1}^x+\psi_{\lfloor tN^{2\alpha}\rfloor-k+1}^y),\hat{\nu}^N_j)\\
 &\leq \sum_{k=1}^{\lfloor tN^{2\alpha}\rfloor}\sum_{j\leq k-1}\left( N^{-1}|x-y|^{\frac{1}{2}}(\lfloor tN^{2\alpha}\rfloor-k+1)^{-\frac{1}{2}}k+N^{\frac{\alpha}{2}}(\lfloor tN^{2\alpha}\rfloor-k+1)^{-\frac{3}{4}}k \right)\\
 &\qquad \qquad \qquad \qquad \qquad \qquad \qquad \qquad \qquad \qquad\cdot(A(\hat{\xi}_{k-1})e_{-\lambda}(\psi_{\lfloor tN^{2\alpha}\rfloor-k+1}^x+\psi_{\lfloor tN^{2\alpha}\rfloor-k+1}^y)),\hat{\nu}^N_j)
\end{align*}
Using (b), (c), (d) of Lemma \ref{psi} and (a), (c) of Lemma \ref{moment},
\begin{equation}
\label{attrition mart est}
\begin{aligned}
(A(\hat{\xi}_{k-1})e_{-\lambda}\psi_{\lfloor tN^{2\alpha}\rfloor-k+1}^x),\hat{\nu}^N_j)&\leq \| A^p(\hat{\xi}_{k-1})\|^{\frac{1}{p}}_{-\lambda p}(\psi_{\lfloor tN^{2\alpha}\rfloor-k+1}^x,\hat{\nu}^N_j)\\
&\leq \| A^p(\hat{\xi}_{k-1})\|^{\frac{1}{p}}_{-\lambda p}\sup_{1\leq j\leq \lfloor tN^{2\alpha}\rfloor}(\e_{-\lambda},\hat{\nu}^N_j)\|\psi_{\lfloor tN^{2\alpha}\rfloor-k+1}^x\|_{\lambda}\\
&\leq  \| A^p(\hat{\xi}_{k-1})\|^{\frac{1}{p}}_{-\lambda p}\sup_{1\leq j\leq \lfloor tN^{2\alpha}\rfloor}(\e_{-\lambda},\hat{\nu}^N_j)\e_{\lambda}(x)N^{\alpha}(\lfloor tN^{2\alpha}\rfloor-k+1)^{-\frac{1}{2}}.
\end{aligned}
\end{equation}
Therefore, by using the fact that $\alpha=1/5$,
\begin{align*}
&~~~~\E\left| M_{\lfloor tN^{2\alpha}\rfloor}^{(2)}(\psi_{\lfloor tN^{2\alpha}\rfloor-\cdot}^x-\psi_{\lfloor tN^{2\alpha}\rfloor\cdot}^y) \right|^p \\
&\leq C(\lambda,p,f,T)\e_{\lambda p}(x)\cdot\left(\sum_{k=1}^{\lfloor tN^{2\alpha}\rfloor}  N^{\alpha-1}|x-y|^{\frac{1}{2}}(\lfloor tN^{2\alpha}\rfloor-k+1)^{-1}k+N^{\frac{\alpha}{2}-1}(\lfloor tN^{2\alpha}\rfloor-k+1)^{-\frac{5}{4}}k \right)^{\frac{p}{2}}\\
&\leq C(\lambda,p,f,T)\e_{\lambda p}(x)\left( N^{\frac{5\alpha-1}{2}p}|x-y|^{\frac{p}{4}}+N^{\frac{2\alpha-1}{2}p} \right)\\
&\leq C(\lambda,p,f,T)\e_{\lambda p}(x)\left( |x-y|^{\frac{p}{4}}+N^{-\frac{3\alpha }{2}p} \right).
\end{align*}
Similarly, for the time difference, we first deal with the drift term
\begin{align*}
&~~~~\sum_{k=1}^{\lfloor tN^{2\alpha}\rfloor}\sum_{j\leq k-1}\frac{1}{N^{1-\alpha}}\left( A(\hat{\xi}_{k-1})\psi^y_{\lfloor tN^{2\alpha}\rfloor-k+1},\hat{\nu}^N_j \right)-\sum_{k=1}^{\lfloor sN^{2\alpha}\rfloor}\sum_{j\leq k-1}\frac{1}{N^{1-\alpha}}\left( A(\hat{\xi}_{k-1})\psi^y_{\lfloor sN^{2\alpha}\rfloor-k+1},\hat{\nu}^N_j \right)\\
&=\sum_{k=1}^{\lfloor sN^{2\alpha}\rfloor}\sum_{j\leq k-1}\frac{1}{N^{1-\alpha}}\left( A(\hat{\xi}_{k-1})\left(\psi^y_{\lfloor tN^{2\alpha}\rfloor-k+1}-\psi^y_{\lfloor sN^{2\alpha}\rfloor-k+1}\right) ,\hat{\nu}^N_j\right)\\
&~~~~+\sum_{k=\lfloor sN^{2\alpha}\rfloor+1}^{\lfloor tN^{2\alpha}\rfloor}\sum_{j\leq k-1}\frac{1}{N^{1-\alpha}}\left( A(\hat{\xi}_{k-1})\psi^y_{\lfloor tN^{2\alpha}\rfloor-k+1},\hat{\nu}^N_j \right).
\end{align*}
By (b), (e) of Lemma \ref{psi},  (c) of Lemma \ref{moment} and the fact that $\alpha=1/5$, the $p$-th moment of the first term above can be bounded by
\begin{align*}
&~~~~\E \left| \sum_{k=1}^{\lfloor sN^{2\alpha}\rfloor}\sum_{j\leq k-1}\frac{1}{N^{1-\alpha}}\left( A(\hat{\xi}_{k-1})\left(\psi^y_{\lfloor tN^{2\alpha}\rfloor-k+1}-\psi^y_{\lfloor sN^{2\alpha}\rfloor-k+1}\right) ,\hat{\nu}^N_j\right) \right| \\
&\leq C(\lambda,p,f,T)\e_{\lambda p}(y)\left( \sum_{k=1}^{\lfloor sN^{2\alpha}\rfloor} N^{\frac{5}{2}\alpha-1}|t-s|^{\frac{1}{2}}(\lfloor sN^{2\alpha}\rfloor-k+1)^{-\frac{3}{4}}k+N^{\frac{3}{2}\alpha-1}(\lfloor sN^{2\alpha}\rfloor-k+1)^{-\frac{3}{4}}k \right)^p\\
&\leq C(\lambda,p,f,T)\e_{\lambda p}(y)\left( N^{(5\alpha-1)p}|t-s|^{\frac{p}{2}}+N^{(4\alpha-1)p} \right)\\
&\leq C(\lambda,p,f,T)\e_{\lambda p}(y)\left( |t-s|^{\frac{p}{2}}+N^{-\alpha p} \right).
\end{align*}
By (b), (c) of Lemma \ref{psi}, (c) of Lemma \ref{moment} and the fact that $\alpha=1/5$, the $p$-th moment of the second term above can be bounded by
\begin{align*}
&~~~~\E \left| \sum_{k=\lfloor sN^{2\alpha}\rfloor+1}^{\lfloor tN^{2\alpha}\rfloor}\sum_{j\leq k-1}\frac{1}{N^{1-\alpha}}\left( A(\hat{\xi}_{k-1})\psi^y_{\lfloor tN^{2\alpha}\rfloor-k+1},\hat{\nu}^N_j \right) \right|^p\\
&\leq C(\lambda,p,f,T)\e_{\lambda p}(y) \left( \sum_{k=\lfloor sN^{2\alpha}\rfloor+1}^{\lfloor tN^{2\alpha}\rfloor}N^{2\alpha-1}(\lfloor tN^{2\alpha}\rfloor-k+1)^{-\frac{1}{2}}k \right)^p\\
&\leq C(\lambda,p,f,T)\e_{\lambda p}(y) \left( N^{(5\alpha-1)p}|t-s|^{\frac{p}{2}} \right)\\
&\leq C(\lambda,p,f,T)\e_{\lambda p}(y)|t-s|^{\frac{p}{2}}.
\end{align*}
To deal with the part of $M^{(2)}_{\lfloor tN^{2\alpha}\rfloor}(\psi^y_{\lfloor tN^{2\alpha}\rfloor-\cdot})-M^{(2)}_{\lfloor sN^{2\alpha}\rfloor}(\psi^y_{\lfloor sN^{2\alpha}\rfloor-\cdot})$, we can separate it into two parts and use BDG inequality.

The first part is $M^{(2)}_{\lfloor sN^{2\alpha}\rfloor}(\psi^y_{\lfloor tN^{2\alpha}\rfloor-\cdot}-\psi^y_{\lfloor sN^{2\alpha}\rfloor-\cdot})$ with quadratic variation
\begin{align*}
&~~~~\langle M^{(2)}(\psi^y_{\lfloor tN^{2\alpha}\rfloor-\cdot}-\psi^y_{\lfloor sN^{2\alpha}\rfloor-\cdot}) \rangle_{\lfloor sN^{2\alpha}\rfloor}\\
&\leq \sum_{k=1}^{\lfloor sN^{2\alpha}\rfloor}\sum_{j\leq k-1} \frac{\|\psi^y_{\lfloor tN^{2\alpha}\rfloor-k+1}-\psi^y_{\lfloor sN^{2\alpha}\rfloor-k+1}\|_{\lambda}}{N^{1+\alpha}}\left(  A(\hat{\xi}_{k-1})e_{-\lambda}(\psi^y_{\lfloor tN^{2\alpha}\rfloor-k+1}+\psi^y_{\lfloor sN^{2\alpha}\rfloor-k+1}),\hat{\nu}^N_j\right)\\
&\leq \sum_{k=1}^{\lfloor sN^{2\alpha}\rfloor}\left( N^{\frac{\alpha}{2}-1}|t-s|^{\frac{1}{2}}(\lfloor sN^{2\alpha}\rfloor-k+1)^{-\frac{3}{4}}k+N^{-1-\frac{\alpha}{2}}(\lfloor sN^{2\alpha}\rfloor-k+1)^{-\frac{3}{4}}k\right)\\
&\qquad \qquad \qquad \qquad \qquad \qquad \qquad \qquad \qquad \qquad\cdot\left(  A(\hat{\xi}_{k-1})e_{-\lambda}(\psi^y_{\lfloor tN^{2\alpha}\rfloor-k+1}+\psi^y_{\lfloor sN^{2\alpha}\rfloor-k+1}),\hat{\nu}^N_j\right).
\end{align*}
Inequality (\ref{attrition mart est}) gives us
\begin{align*}
&~~~~\E \left| M^{(2)}_{\lfloor sN^{2\alpha}\rfloor}(\psi^y_{\lfloor tN^{2\alpha}\rfloor-\cdot}-\psi^y_{\lfloor sN^{2\alpha}\rfloor-\cdot}) \right|^p \\
&\leq C(\lambda,p,f,T)\e_{\lambda p}(y)\cdot\left(\sum_{k=1}^{\lfloor sN^{2\alpha}\rfloor} N^{\frac{3}{2}\alpha-1}|t-s|^{\frac{1}{2}}(\lfloor sN^{2\alpha}\rfloor-k+1)^{-\frac{5}{4}}k+N^{\frac{\alpha}{2}-1}(\lfloor sN^{2\alpha}\rfloor-k+1)^{-\frac{5}{4}}k\right)^{\frac{p}{2}}\\
&\leq C(\lambda,p,f,T)\e_{\lambda p}(y)\left( N^{\frac{3.5\alpha-1}{2}p}|t-s|^{\frac{p}{4}}+N^{\frac{2.5\alpha-1}{2}p}\right)\\
&\leq C(\lambda,p,f,T)\e_{\lambda p}(y)\left( N^{-\frac{3\alpha}{4}p}|t-s|^{\frac{p}{4}}+N^{-\frac{5\alpha}{4}p} \right).
\end{align*}
The second part is $M^{(2)}_{\lfloor tN^{2\alpha}\rfloor}(\psi^y_{\lfloor tN^{2\alpha}\rfloor-\cdot})-M^{(2)}_{\lfloor sN^{2\alpha}\rfloor}(\psi^y_{\lfloor tN^{2\alpha}\rfloor-\cdot})$ with quadratic variation
\begin{align*}
~~~~&\langle M^{(2)}(\psi^y_{\lfloor tN^{2\alpha}\rfloor}-\cdot)\rangle_{\lfloor tN^{2\alpha}\rfloor}-\langle M^{(2)}(\psi^y_{\lfloor tN^{2\alpha}\rfloor}-\cdot)\rangle_{\lfloor sN^{2\alpha}\rfloor}\\
&\leq \sum_{k=\lfloor sN^{2\alpha}\rfloor+1}^{\lfloor tN^{2\alpha}\rfloor}\sum_{j\leq k-1}\frac{\|\psi^y_{\lfloor tN^{2\alpha}\rfloor-k+1}\|_{\lambda}}{N^{1+\alpha}}\left( A(\hat{\xi}_{k-1})e_{-\lambda}\psi^y_{\lfloor tN^{2\alpha}\rfloor-k+1},\hat{\nu}^N_j \right)\\
&\leq \sum_{k=\lfloor sN^{2\alpha}\rfloor}^{\lfloor tN^{2\alpha}\rfloor}\sum_{j\leq k-1}N^{-1}(\lfloor tN^{2\alpha}\rfloor-k+1)^{-\frac{1}{2}}\left( A(\hat{\xi}_{k-1})e_{-\lambda}\psi^y_{\lfloor tN^{2\alpha}\rfloor-k+1},\hat{\nu}^N_j \right).
\end{align*}
Inequality (\ref{attrition mart est}) again gives us
\begin{align*}
&~~~~\E \left| M^{(2)}_{\lfloor tN^{2\alpha}\rfloor}(\psi^y_{\lfloor tN^{2\alpha}\rfloor-\cdot})-M^{(2)}_{\lfloor sN^{2\alpha}\rfloor}(\psi^y_{\lfloor tN^{2\alpha}\rfloor-\cdot}) \right|^p \\
&\leq C(\lambda,p,f,T)\e_{\lambda p}(y)\left( \sum_{k=\lfloor sN^{2\alpha}\rfloor+1}^{\lfloor tN^{2\alpha}\rfloor} N^{\alpha-1}(\lfloor tN^{2\alpha}\rfloor-k+1)^{-1}k \right)^{\frac{p}{2}}\\
&\leq C(\lambda,p,f,T)\e_{\lambda p}(y)\left( N^{\frac{5\alpha-1}{2}p}|t-s|^{\frac{p}{2}} \right)\\
&\leq C(\lambda,p,f,T)\e_{\lambda p}(y) |t-s|^{\frac{p}{2}}.
\end{align*}
Combining with Lemma \ref{tight}, we get (\ref{eqn: true tightness}).
\end{proof}
The tightness of $A(\hat{\xi}_{\lfloor tN^{2\alpha}\rfloor})$ follows from Lemma \ref{real proc tight}, which means that we can find a subsequence with a limit $\hat{u}_t$. Since the true process is dominated by the branching envelope, we easily see that Lemma \ref{measure density} also holds for the true horizontal process. This implies the tightness of $\hat{\nu}^N_{\lfloor tN^{2\alpha}\rfloor}$ under vague topology. Let $\hat{\nu}_t$ be a weak limit. By substituting $\phi^N_k=\phi\in C^2_0(\R)$ in the semimartingale decomposition (\ref{real proc decomp}) and Lemma \ref{small error}, if $\alpha=1/5$, we can see that the martingale $\hat{M}^N_{\lfloor tN^{2\alpha}\rfloor}$ can be written as
\begin{align*}
\hat{M}^N_{\lfloor tN^{2\alpha}\rfloor}(\phi)=(\hat{\nu}^N_{\lfloor tN^{2\alpha}\rfloor},\phi)-(A(\hat{\xi}_0),\phi)-\sum_{k=1}^{\lfloor tN^{2\alpha}\rfloor}(\hat{\nu}^N_k,\Delta_D \phi)-\sum_{k=1}^{\lfloor tN^{2\alpha}\rfloor}\sum_{j\leq k-1}\frac{1}{N^{1-\alpha}}(A(\hat{\xi}_{k-1}\phi),\hat{\nu}^N_j)+O(N^{-2/5})
\end{align*}
and every term on the right-hand side converges almost surely by Lemma \ref{real proc tight}. Hence $\hat{M}^N_{\lfloor tN^{2\alpha}\rfloor}(\phi)$ converges to a local martingale
\begin{equation}
\label{true decomp cont}
\begin{split}
\hat{m_t}(\phi)&=(\hat{\nu}_t,\phi)-(\hat{\nu}_0,\phi)-\frac{1}{6}\int_0^t (\hat{\nu}_s,\Delta\phi)ds-\int_0^t\left( \hat{\nu}_s,\int_0^s \hat{u}_r\phi\right)ds\\
&=\int \phi(x)\hat{u}_t(x)dx-\int \phi(x)f(x)dx-\frac{1}{6}\int_0^t \int \Delta \phi(x)\hat{u}_s(x)dxds-\int_0^t\int_0^s \int \phi(x) \hat{u}_s(x)\hat{u}_r(x)dxdrds,
\end{split}
\end{equation}
which is continuous since every term on the right-hand side is continuous. Moreover, from (\ref{real proc bracket}),
\begin{align*}
(\hat{M}^N_{\lfloor tN^{2\alpha}\rfloor})^2-\sum_{k=1}^{\lfloor tN^{2\alpha}\rfloor}\frac{1}{N^{2\alpha}}(A(\hat{\xi}_{k-1}),\phi^2)\left(1-\frac{1}{2N}\right)-\sum_{k=1}^{\lfloor tN^{2\alpha}\rfloor}\sum_{j\leq k-1}\frac{1}{N^{1+\alpha}}(A(\hat{\xi}_{k-1})\phi^2,\hat{\nu}^N_j)\left(1-\frac{1}{2N}\right)
\end{align*}
is a martingale. As $N\rightarrow \infty$,
\begin{equation}
\label{true quad mart}
\hat{m}_t^2(\phi)-\int_0^t\int \phi^2(x)\hat{u}_s(x)dxds
\end{equation}
is also a continuous local martingale. (\ref{true decomp cont}) and (\ref{true quad mart}) prove that any subsequential weak limit $\hat{\nu}_t(dx)=\hat{u}_t(x)dx$ solves (\ref{real spde}).

\subsection{Girsanov transformation. Proof of the uniqueness in Theorem \ref{thm real spde}}
\label{sec: uniqueness}
As is discussed in Section \ref{tightness}, the envelope measure $\nu_t$ solves the martingale problem: $\forall \phi\in C_0^2(\mathbb{R})$ test function twice differentiable with compact support, the process
$$m_t(\phi)=(\nu_t,\phi)-(\nu_0,\phi)-\frac{1}{6}\int_0^t (\nu_s, \Delta \phi)ds$$
is a continuous local martingale with quadratic variation process
$$\langle m(\phi)\rangle_t=\int_0^t (\nu_s,\phi^2)ds.$$
From this, we know that
\begin{align*}
\e^{-(\nu_t,\phi)}-\e^{-(\nu_0,\phi)}-\int_0^t \e^{-(\nu_s,\phi)}\left( \nu_s,-\frac{1}{6}\Delta \phi+\phi^2 \right)ds
\end{align*}
is a continuous local martingale. Using the duality method in Section 4.4 of \cite{ek}, we can choose triplet $(h,0,0)$ on the space $\mathcal{M}_F\times C_0^2$, where $\mathcal{M}_F$ is the collection of finite Borel measures and $h(\cdot,\cdot)$ is defined as $$h(\nu,\phi)=\e^{-(\nu,\phi)}.$$
Then $$\E h(\nu_t,\phi)=h(\nu_0,u_t^*),$$ $u_s^*$ is the solution to the deterministic equation
\begin{align}
\label{dual eqn}
\begin{cases}
&\frac{\partial u_t^*}{\partial t} = \frac{1}{6}\Delta u_t^*-(u_t^*)^2  \\
 & u_0^* = \phi .
\end{cases}
\end{align}
$\{u_t^*\}_{t\geq 0}$ is the dual process of the solution to the martingale problem. The existence of solution to (\ref{dual eqn}) gives the uniqueness of $\{ \nu_t \}_{t\geq 0}$.

Let $m(ds,dx)$ be the orthogonal martingale measure of $m_t(\cdot)$, which means that it is of intensity measure $$\nu((0,t]\times A)=\int_A \int_0^t u_s(x)dsdx,$$ for any Borel measurable set  $A\subset\mathbb{R}$. Then the Radon-Nykodym derivative of the true process with respect to the envelope is
\begin{align}
\label{r-n derivative}
\left.\frac{d\Q}{d\mathbb{P}}\right|_t=\exp \left\{ -\int \int_0^t \theta(s,x)m(ds,dx)-\frac{1}{2} \int_0^t (u_s,\theta(s,\cdot)^2)ds \right\},
\end{align}
where the drift term
$$\theta(s,x)=\int_0^s u_r(x)dr.$$
The uniqueness of $\{ \hat{\nu}_t\}_{t\geq 0}$ follows directly from the uniqueness of $\{ \nu_t \}_{t\geq 0}$. This concludes the proof of Theorem \ref{thm real spde}.

\section{Existence of Percolation}
\label{sec:existence}
In the past two sections, we have shown that $\alpha=1/5$ (in the sense of Theorem \ref{thm real spde}) is a critical exponent for the horizontal process. The envelope process on each horizontal layer follows the law with asymptotic approximate density  given by the solution of (\ref{eqn: brw spde}). In the anisotropic percolation model, the horizontal movement has an attrition compared to the envelope process. The attrition comes from two parts:
\begin{itemize}
\item In the envelope process, it is allowed to have multiple particles at each site. However, in the true mechanism, we only consider if a site is occupied or not hence the configuration at each site can only take values in  $0$ or $1$. Fortunately, the probability of multiple particles is negligible when $\alpha=1/5$ (Corollary \ref{cor: multiple}).
\item As was explained in the Introduction, the vertical interaction should be only considered once for any site in the anisotropic percolation. When we consider the horizontal movement, any site that has been visited before cannot be visited again. Under the critical exponent $\alpha=1/5$, this attrition becomes significant and leads to the part
$$-\hat{u_t}\int_0^t\hat{u}_sds$$
 in the asymptotic approximate density.
\end{itemize}
In this section we prove Theorem \ref{thm:main}, by investigating the occurrence (or not) of percolation when on each layer we have the true model, and the vertical bonds between neighbouring sites are open with probability $p_v=\kappa N^{-2/5}$, all independently.

\subsection{The case $\kappa<C_1$}
\label{sec:sub}
As we have discussed in Section \ref{intro}, the occupied sites at each layer follow a horizontal process with attrition whose asymptotic approximate density follows the SPDE (\ref{real spde}). Here we abuse the notation $\mathcal{C}^i_x$ as the cluster starting from $x$ at layer $i$. in the rescaled space $\Z/N^{6/5}\times \Z$. The main theorem to show in this subsection is as follows.
\begin{theorem}
\label{thm:sub}
For the horizontal process with attrition, there exists a constant $L$ such that the cumulated number of occupied sites (or the cluster size) starting from zero satisfies
$$\E |\mathcal{C}^0_0|\leq LN^{2/5}.$$
\end{theorem}

Before proving the main theorem, let us show how it implies that there is no percolation when $\kappa<C_1$ for $C_1$ small enough.
\begin{corollary}
\label{cor:sub}
Let $p_v=\kappa N^{-2/5}$ denote the probability of a vertical edge being open. There exists $C_1$ such that for $\kappa<C_1$, there is no percolation in the anisotropic percolation system for all $N$ large.
\end{corollary}
\begin{proof}[Proof of Corollary \ref{cor:sub}]
Recall that the horizontal edges, i.e. edges between $(x,i)$ and $(y,i)$ for some $i$ and $x\sim y$, are open with probability $1/(2N)$, while the vertical ones between $(x,i)$ and $(x,j)$ for some $x$ and $|j-i|=1$ are open with probability $p_v$, all independently. 
We say that there is a path from $(x,i)$ at layer $i$ to $(y,k)$ at layer $k$ denoted by $(x,i)\rightarrow (y,k)$ if there is $n$ and $x_j,i_j,1\leq j\leq n$ so that $(x_1,i_1)=(x,i)$, $(x_n,i_n)=(y,k)$ and $\forall 1\leq j\leq n-1$, the edge between $(x_j,i_j)$ and $(x_{j+1},i_{j+1})$ is open.

We want to explore all sites that are connected to $(0,0)$, i.e. that can be reached by an open path from $(0,0)$. Once an open path reaches layer $i$, it can continue through vertical neighbours at layers $i\pm 1$, moving upward or downward; we can count the number of connected sites with a certain number of vertical movements from layer $0$ rather than its layer number.

After $n$ movements which contain $m$ vertical movements (upward or downward), there is a collection of open paths from the origin $(x_0,i_0)=(0,0)\rightarrow(x_n,i_n)$. Let $I_v\subset{1,2,\cdots,n}$ be the set of vertical movements such that $|I_v|=m$ and $\forall k\in I_v,|i_k-i_{k-1}|=1,x_k=x_{k-1}$. For $k\in \{1,2,\cdots,n\}\backslash I_v$ i.e. the horizontal movement indices, $i_k=i_{k-1}, x_k\sim x_{k-1}$. Denote $\mathcal{S}_m$ as the collection of points which are the ends of these paths from the origin after $m$ vertical movements (with any number $n\geq m$ of total movements). 

These sites are possibly to be distributed on different layers.
In the development of $\{\mathcal{S}_m\}_{m\geq 0}$, we consider the horizontal movements and vertical movements separately at each time. More precisely (ref. Figure \ref{fig:sub}), we start with $(0,0)$, and following the law $\mathcal{C}^0_0$ we produce connected sites at layer $0$. In the first vertical movement, these sites at layer $0$ can connect to sites at layer $\pm1$. Before the second vertical movement, these connected sites at layers $\pm1$ will produce an horizontal cluster following the law of $\mathcal{C}^0_0$ at its layer, which will then connect to sites at layers $\pm2$ and $0$. $\mathcal{S}_m$ can be constructed inductively by considering the total number of horizontal connected sites and then their vertical movements.

\begin{figure}
\center
\includegraphics[scale=0.7]{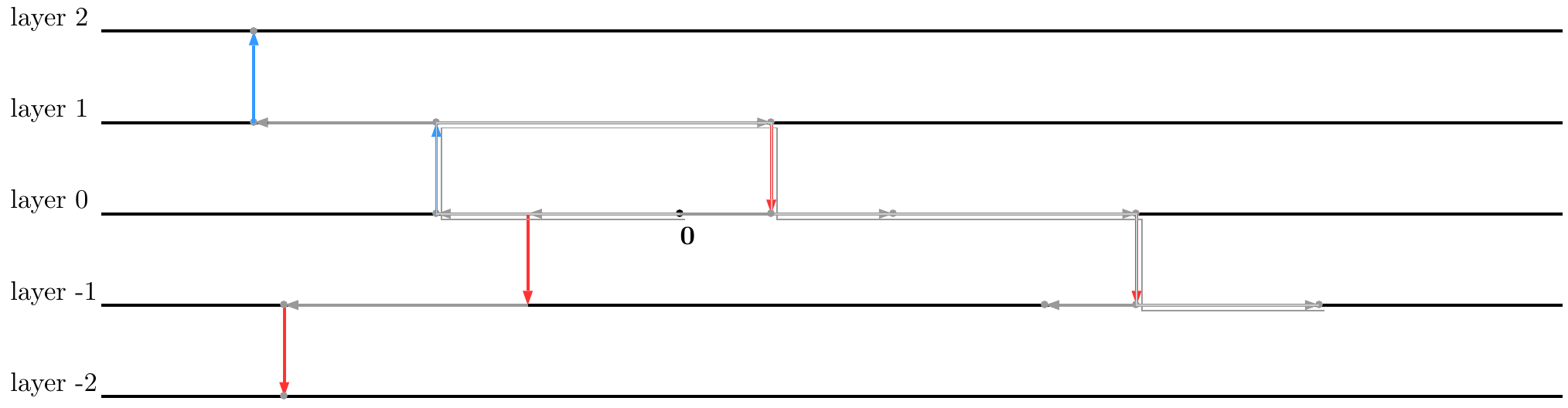}
\caption{Movement of $\mathcal{S}_m$}
\label{fig:sub}
\end{figure}

Due to attrition, in the horizontal connection we only consider a site to be occupied or not, rather than the number of particles at each site,
the cardinality $\{|\mathcal{S}_m|\}_{m \geq 0}$ is stochastically dominated (in the sense of Definition II.2.3 of \cite{liggett}) by a branching process $\{Z_m\}_{m\geq 0}$ following the law
\begin{align*}
Z_{m+1}=\sum_{i=1}^{Z_m}Y_{m,i}, &\text{ where }Y_{m,i} \sim\text{Binomial}(2\mathcal{N}_{m,i},p_v) \text{ for }1\leq i\leq Z_m,\\
&\mathcal{N}_{m,i}\text{ is independent of }Z_1,\cdots,Z_m \text{ for each }i,m\\
&\text{and }(\mathcal{N}_{m,i})_{m,i}\text{ is an i.i.d. sequence with distribution as }|\mathcal{C}_0^0|.
\end{align*}

Theorem \ref{thm:sub} gives the upper bound of $\E|\mathcal{C}_0^0|$ to be $2LN^{2/5}$.
When $\kappa$ is small enough to make $2\kappa L<1$, $\{Z_m\}_{m\geq 0}$ is a sub-critical branching process which will die out (ref. Theorem A.5.1 of \cite{branching}). Therefore, there exists positive constant $C_1=(2L)^{-1}$ such that for $\kappa<C_1$, there is no percolations in this layered system.
\end{proof}

We now move to the proof of Theorem \ref{thm:sub}.  For this we will need two inequalities (Lemma \ref{finite Tk discrete} and Proposition \ref{prop 1/8} below) which concern the following hitting times for the branching envelope and for the true horizontal process:
\begin{align*}
\tilde{T}_k=\inf\left\{n: \sum_x \xi_n(x)\geq 2^k \text{ or }\sum_{i=1}^n\sum_x \xi_i(x)\geq 2^{2k}\right\},
\end{align*}
which is just the discrete version of the hitting time $T_k$ in Section \ref{sec:multiple}, and 
\begin{align*}
\hat{T}_k=\inf\left\{n: \sum_x \hat{\xi}_n(x)\geq 2^k \text{ or }\sum_{i=1}^n\sum_x \hat{\xi}_i(x)\geq 2^{2k}\right\}=\hat{T}'_k\wedge \hat{T}''_k,
\end{align*}
where
\begin{align*}
\hat{T}'_k=\inf\{n:\sum_x \hat{\xi}_n(x)\geq 2^k\}, \hat{T}''_k=\inf \left\{n:\sum_{i=0}^n \sum_x \hat{\xi}_i(x)\geq 2^{2k}  \right\}.
\end{align*}

\begin{lemma}
\label{finite Tk discrete}
Suppose $\xi_0(x)=1$ if $x=0$ and $\xi_0(x)=0$ otherwise, then we have
$$\P(\tilde{T}_k<\infty)<C2^{-k}.$$
\end{lemma}
\begin{proposition}
\label{prop 1/8}
Let integer $k_0$ be defined by $2^{k_0}\leq N^{2/5}<2^{k_0+1}$. There exists $M_1$ large such that for any $k=k_0+\log_2 M_1+r,r\geq 0$, we have
$$\P(\hat{T}_k<\infty\mid \hat{T}_{k_0+\log_2M_1}<\infty)\leq \frac{C}{8^r}.$$
\end{proposition}
Postponing the proofs of these estimates, we first see how they allow us to conclude the proof of Theorem \ref{thm:sub}.

\begin{proof}[Proof of Theorem \ref{thm:sub}]
The proof  is given in the following steps. As we can see in the proof of Theorem \ref{real spde}, the attrition part is negligible when $\alpha<1/5$ becomes significant when $\alpha=1/5$. Because of attractiveness, we only need to consider the attrition once the total mass is of order $O(N^{2/5})$. So we consider a process that dominates the horizontal process, which follows the pure branching random walk before the total mass reaches $M_1N^{2/5}$ for some $M_1$  large and includes the attrition part after that. We are first interested in the crossing time of $\sum_x\xi_n(x)$ over level $M_1N^{2/5}$. 

The dominating process that we consider in this subsection follows $\{\xi_k(x)\}$ before $\tilde{T}_{k_0+\log_2 M_1}$ and follows $\{\hat{\xi}_k(x)\}$ after $\tilde{T}_{k_0+\log_2 M_1}$.
The reason of separating the time is as follows. The size of cluster containing the origin satisfies
\begin{align*}
\E |\mathcal{C}^0_0| &\leq \sum_{k=0}^{\infty}2^{2(k+1)}\P(\hat{T}_k<\infty)\\
&\leq \sum_{k=0}^{k_0+\log_2 M_1}2^{2(k+1)} \P(\tilde{T}_k<\infty)+\sum_{k\geq k_0+\log_2 M_1}2^{2(k+1)}\P(\hat{T}_{k+1}<\infty)\\
&\leq \sum_{k=0}^{k_0+\log_2 M_1}2^{2(k+1)}C2^{-k}+\sum_{k\geq k_0+\log_2 M_1}2^{2(k+1)}\P(\hat{T}_{k+1}<\infty)\\
&\leq 8CM_1N^{2/5}+\sum_{k\geq k_0+\log_2 M_1}2^{2(k+1)}\P(\hat{T}_{k+1}<\infty).
\end{align*}
The third inequality is by Lemma \ref{finite Tk discrete} and the fourth inequality is due to the fact that $2^{k_0}\leq N^{2/5}$. The last work is to bound the second term in the last inequality. By Proposition \ref{prop 1/8},
the size of cluster containing zero can be bounded by
\begin{align*}
\E |\mathcal{C}^0_0| &\leq 8CN^{2/5}+\sum_{k\geq k_0+\log_2 M_1}2^{2k+2}C2^{-k_0-\log_2 M_1}8^{-(k-k_0-\log_2 M_1)}\\
&\leq 8CM_1N^{2/5}+8CM_1N^{2/5}.
\end{align*}
This finishes the proof of Theorem \ref{thm:sub}.
\end{proof}
In the following part of this subsection, we will show Lemma \ref{finite Tk discrete} and Proposition \ref{prop 1/8}.

\begin{proof}[Proof of Lemma \ref{finite Tk discrete}]
The proof is similar as in Lemma \ref{finite Tk}. It is followed by replacing the corresponding part in the proof of Lemma \ref{finite Tk} that
$$\E \left( (X_t-X_s)^2\mid X_s \right)=\E \left( \int_s^t X_udu \middle| X_s \right)$$
into the fact that

\begin{align*}
\E \left( \left(\sum_x\xi_n(x)-\sum_x\xi_m(x)\right) \middle| \mathcal{F}_m \right) &=\left(1-\frac{1}{2N}\right)\E \left( \sum_{j=m}^{n-1} \sum_x\xi_j(x) \middle| \mathcal{F}_m\right).
\end{align*}

Then we can use the similar martingale technique and the fact that
\begin{equation}
\label{eqn:discrete mass}
\sum_x \xi_{n+1}(x)=\sum_x \xi_n(x)+\sum_x\sum_{y\sim x}\sum_{w=1}^{\xi_n(y)}\left( \eta_{n+1}^w(y,x)-\frac{1}{2N} \right)
\end{equation}
is a martingale. Denote the discrete mass as $$\tilde{X}_n=\sum_x \xi_n(x).$$
The desired probability can be decomposed as
\begin{align*}
\P(\tilde{T}_k<\infty) &= \P(\tilde{T}'_k<\infty,\tilde{T}'_k<\tilde{T}''_k)+\P(\tilde{T}''_k<\infty,\tilde{T}''_k<\tilde{T}'_k)\\
&\leq \P(\tilde{T}'_k<\infty)+\P(\tilde{T}''_k<T'_k),
\end{align*}
where
\begin{align*}
\tilde{T}'_k=\inf\{n:\sum_x \xi_n(x)\geq 2^k\}, \tilde{T}''_k=\inf \left\{n:\sum_{i=0}^n \sum_x \xi_i(x)\geq 2^{2k}  \right\}.
\end{align*}
Denote $\tilde{H}_0=\inf\{n:\tilde{X}_n=0\}$ as the first hitting time of zero, then $\P(\tilde{T}'_k<\infty)=\P(\tilde{T}'_k<\tilde{H}_0)$, and this is simply
$$\P(\tilde{T}'_k<\infty)\leq \frac{\tilde{X}_0}{2^k}=\frac{1}{2^k}.$$
The event $\{\tilde{T}''_k<\tilde{T}'_k\}\subset \bigcup_{j=1}^{k-1}A_j$, where
$$A_j=\left\{ \tilde{T}'_j<\infty,\sum_{i=\tilde{T}'_j+1}^{\tilde{T}'_{j+1}}\tilde{X}_i\geq \frac{6\cdot 2^{2k}}{\pi (k-j)^2} \right\}.$$
For $m<n$, we have
\begin{align*}
&~~~~\E \left[ (\tilde{X}_{n}-\tilde{X}_m)^2\mid \tilde{X}_m \right]\\
&=\E\left[ \left(\sum_{i=m}^{n-1}(\tilde{X}_{i+1}-\tilde{X}_i)\right)^2\mid \tilde{X}_m \right]\\
&=\E \left[ \sum_{i=m}^{n-1}(\tilde{X}_{i+1}-\tilde{X}_i)^2\mid \tilde{X}_m \right]\\
&=\left(1-\frac{1}{2N}\right)\E\left[\sum_{i=m}^{n-1}\tilde{X}_i\mid \tilde{X}_m\right],
\end{align*}
where the second equality follows at once from the definition of $\{\tilde{X}_k\}_{k\geq0}$ (a branching process).
Letting $n=\tilde{T}'_{j+1}$ and $m=\tilde{T}'_j$ gives that
\begin{align*}
\P(A_j)&\leq \frac{(2^{j+2}-2^j)^2\cdot \pi(k-j)^2}{2^{2k}},
\end{align*}
by the strong Markov property of $\{\tilde{X}_k)_{k\geq0}$ at stopping time $\tilde{T}'_j$.
Therefore,
\begin{align*}
\P(\tilde{T}_k<\infty)&\leq \frac{1}{2^k}+\sum_{j=1}^{k-1}\P(A_j)\\
&\leq \frac{C}{2^k}.
\end{align*}
\end{proof}

The above proof immediately yields
\begin{corollary} \label{tkcor}
Given a stopping time $T$ with respect to the natural filtration of the $\{\xi_n(x), n \ge 0, x \in \Z / N^{6/5}\}$,
the stopping time $T(k) = \inf \left\{n \geq T: \sum_x\xi_n(x) \ge 2^k \mbox{ or } \sum_{m=0}^n\sum_x\xi_m(x) \geq 2^{2k} \right\} $ satisfies
$$
\P( T(k) < \infty \vert \F _T ) \leq  C 2^{-r}
$$
on the set
\begin{align*}
\left\{\tilde{X}_T \leq 2^{k-r}, \sum_{m=0}^T \tilde{X}_m \leq 2^{2k} /2\right\}
\end{align*}
for universal finite $C$ (uniform in $N$) and integer $r$.
\end{corollary}

To show Proposition \ref{prop 1/8}, we need two properties of the branching processes: on the large deviations and the next one is on the population size of the critical branching process. 

\begin{lemma}
\label{large deviation Tk}
For a sequence of random variables $Y_i\overset{\text{i.i.d.}}{\sim}\text{Binomial}(2N,1/(2N)),i=1,\dots,n$, we have for $a>\frac{1}{2}$,
$$\P\left( \sum_{i=1}^n(Y_i-1)\geq n^{a} \right)\leq e^{-cn^{a\wedge(2a-1)}}.$$
\end{lemma}

\begin{proof}
The proof is followed by large deviation technique.
\begin{align*}
\P\left( \sum_{i=1}^n(Y_i-1)\geq n^{a} \right) &=\P\left( e^{t\sum_{i=1}^n Y_i}\geq e^{(n+n^{a})t} \right)\\
&\leq \exp\left( -(n+n^{a})t+2Nn\log\left( 1+\frac{e^t-1}{2N} \right) \right),\forall t\in\R,
\end{align*}
by the Markov inequality.
The r.h.s. reaches the minimum if $t$ satisfies $$\frac{ne^t}{1+\frac{e^t-1}{2N}}=n+n^{a}.$$ From this, $t=\log(1+n^{a-1})-\log(1-\frac{n^{a-1}}{2N-1})$. If $\frac{1}{2}<a<1$, $t\approx n^{a-1}$ and hence
$$-(n+n^{a})t+2Nn\log\left( 1+\frac{e^t-1}{2N} \right)\leq -cn^{2a-1}.$$
But for $a\geq 1$, $e^t\approx n^{a-1}$ and
$$-(n+n^{a})t+2Nn\log\left( 1+\frac{e^t-1}{2N} \right)\leq -cn^{a}.$$

\end{proof}

\begin{lemma}
\label{binomial critical}
Denote the critical binomial branching process as $\{Z^N_n\}_{n\geq 0}$ with $Z^N_0=1$ and $$Z^N_{n+1}=\sum_{i=1}^{Z^N_n}Y^{n+1}_i,$$
where $(Y^{n+1}_i)_{n,i}$ is an i.i.d. sequence with distribution $\text{Binomial}(2N,1/(2N))$. 
\begin{enumerate}[label=(\roman*)]
\item Given any $T>1$, 
$$\P\left(\frac{2Z^N_{\lfloor tN^{2/5}\rfloor}}{tN^{2/5}}>x\mid Z^N_{\lfloor tN^{2/5}\rfloor}>0\right)\rightarrow e^{-x}$$
as $N\rightarrow\infty$ uniformly in $t\in[1/T,T]$.
\item $(tN^{2/5})\cdot\P(Z^N_{\lfloor tN^{2/5}\rfloor}>0)\rightarrow 2$ as $N\rightarrow\infty$.
\end{enumerate}
\end{lemma}

\begin{proof}
The moment generating function of $Z^N_n$, $f^N_n(t)=\E[t^{Z^N_n}]$ is given by
$$f^N_{n+1}(t)=f^N(f^N_n(t)), f^N(t)=f^N_1(t)=\left(1+\frac{t-1}{2N}\right)^{2N}$$
with $\sup_N(f^N)'''(1)$ bounded. The assertion follows from Theorem I.10.1 of Harris \cite{harris} and the proof therein,
after observing that
$$(f^N)''(1)=\frac{2N(2N-1)}{(2N)^2}\rightarrow 1\text{ as }N\rightarrow\infty$$
hence the result is uniform in $t$.
\end{proof}

With the help of these two properties, we can prove Proposition \ref{prop 1/8} used in the proof of Theorem \ref{thm:sub}.

We first show how Proposition \ref{prop 1/8} follows from the following.

\begin{proposition} \label{extraprop}
Suppose $\hat\xi_0(x)=1$ if $x=0$ and $\hat\xi_0(x)=0$ otherwise, and $\delta > 0$, then there exists $M_3$ sufficiently large that 
$$
\E\left( \hat{X}_{M_3 N^{2/5}} \right) < \delta,
$$
for all $N$, where $\hat{X}_n=\sum_x\hat{\xi}_n(x)$ is the discrete mass of the true horizontal process.
\end{proposition}
\begin{remark}
To summarize the notations about the total mass, $\{X_t\}_{t\geq0}$ denotes the total mass of the envelope process in continuous time given the initial condition $f$ to be continuous, $f(x)=1$ for $x\in[-1,1]$ and compact supported. $\{\tilde{X}_n\}_{n\geq 0}$ and $\{\hat{X}_n\}_{n\geq0}$ represent the total mass of the envelope process and the true horizontal process in discrete time given the initial condition to be $\indic_{\{0\}}$.  
\end{remark}
This Proposition will be proven later.
\begin{proof}[Proof of Proposition \ref{prop 1/8}]
We note that it is sufficient to show that (with $M_1$ chosen sufficiently large)
$$
P(\hat {T}_{k+1} < \infty \vert \hat {T}_{k}  < \infty )< \frac{1}{8}
$$
for $k \geq k_0$. Lemma \ref{large deviation Tk} shows that outside probability 
$e ^{-2^{k/3}}$, we have that
$\hat X_{ \hat {T}_{k}} < 2^k + 2^{2k/3}$ and $\sum_{m=0}^{ \hat {T}_{k}} \hat X _m < 2^{2k}+2^{2k/3}$.

Let $B_1$ be the event that one of these two bounds fails (so $\P(B_1) \leq e ^{-2^{k/3}} < 1/32$ supposing that $M_1$ is sufficiently large).  
We fix $ \delta > 0 $ (to be bounded when needed) and let $M_3$ correspond to $\delta $ in Proposition \ref{extraprop}.
Let $B_2$ be the event that $\sup _{\hat {T}_{k} \leq i \leq \hat {T}_{k} + M_3N^{2/5} } \hat X_i \geq M_3 2^k$.
So by the martingale properties of the envelope process $\P(B_2 \backslash B_1) < \frac{2}{M_3} < 1/32 $ supposing, as we may have that
$M_3$ is sufficiently large.
We note that  on the complement of $B_1 \cup B_2, \sum_{i=0}^{ \hat {T}_{k} +M_3N^{2/5} } \hat{X}_i \leq  2^{2k}+2^{2k/3}
+M_3N^{2/5}  M_3 2^k < 2^{2(k+1)} /2 $ if $M_1 $ is chosen so that $M_1 > 8M_3^2 $ and $N$ is sufficiently large.
Next we have by Proposition \ref{extraprop} and obvious monotonicity
$$
\E( \hat X_{ \hat {T}_{k}+M_3N^{2/5}}\indic_{(B_1\cup B_2)^c}) < (2^k + 2^{2k/3})\delta
$$
and so by the Markov's inequality, the event $B_3 =\{ \hat X_{ \hat {T}_{k}+M_3N^{2/5}}\indic_{(B_1\cup B_2)^c} \geq 2 \sqrt \delta 2^k \}$
has probability bounded by $\sqrt \delta < 1 / 32 $ if $\delta$ was fixed sufficiently small.  
Finally we can apply Corollary \ref{tkcor} to see that $B_4=\{\hat {T}_{k+1} 	< \infty \} \backslash (B_1 \cup B_2 \cup B_3)$ has probability $\P(B_4) < 1/32$ (again supposing $\delta $ to have been fixed sufficiently small).
\end{proof}

In the proof above, we have that at time $\hat{T}_k$, there are around $2^k$ particles. For the process starting from each single one, we want to show that after $M_3N^{2/5}$ steps, some killing property can help to reduce the quantity to be $\delta$ small. It remains to prove Proposition \ref{extraprop}. 

\begin{proof}[Proof of Proposition \ref{extraprop}]
We suppose that $\hat{X} $ is coupled with a envelope process $Z$ so that $\hat{X}_{n}$ is dominated by 
$Z_n$ for each $n$.
We suppose that $\delta>0$ is fixed.  We wish to partition $Z_{N^{2/5}} \ne \ 0$ into sets 
$B_1, B_2,$ and $B_3$ to show (with $M_3$ fixed large that for each $k=1,2,3$, $\E(\hat X_{M_3 N^{2/5}}\indic_{B_k}) < \delta / 4)$.
 Let 
$$\sigma=\inf\left\{n\geq  N^{2/5}/2:\hat{X}_n\leq\epsilon N^{2/5}\right\}$$
where $ \epsilon $ is a small positive constant which remains to be fully specified.  Let
$B_1$ = $\{ \sigma \leq N^{2/5} \}$.   Then by the Strong Markov property applied at $\sigma $
and the martingale property for the envelope of the process
\begin{align*}
\E(\hat X_{M_3 N^{2/5}}\indic_{B_1}) &\leq \P(\hat X_{N^{2/5}}\indic_{B_1} \ne \ 0) \E \left[ \E(\hat X_{M_3 N^{2/5}}\indic_{B_1} \vert \cal{F}_{\sigma}) \right]\\
&\leq  \P(\hat X_{N^{2/5}}\indic_{B_1} \ne \ 0)\epsilon N^{2/5} \\
&\leq \delta / 4,
\end{align*}
if $\epsilon $ was chosen sufficiently small
by Lemma \ref{binomial critical}.

We next consider $B_2 = \{ R_{N^ {2/5}} < 1 / \epsilon \}$ where $R_{N^ {2/5}}$ is the maximal absolute displacement from $0$ of the critical branching random walk $Z$ by time $N^ {2/5}$, i.e.
\begin{equation}
\label{eqn:max envelope}
R_n=\max_{m\leq n}\left\{x\in\Z/N^{6/5}:\xi_m(x)\neq 0\right\}.
\end{equation}

Theorem 1.1 of Kesten \cite{kesten} showed that 
\begin{equation}
\label{eqn: maximal displacement}
\P(R_{N^{2/5}}\geq z\mid Z_{N^{2/5}/2}>0)\leq C_a z^{-a},\text{ for any }a>0,
\end{equation}
(It is easy to check the bound holds for uniformly over $N$). 
Thus  since $Z_{N^{2/5}} / N^{2/5} $ conditioned on being nonzero is uniformly integrable (again using Lemma \ref{binomial critical} ), we have that (again supposing that $\epsilon $ is fixed small)
$$
\E(\hat X_{M_3 N^{2/5}}\indic_{B_2}) \leq \E( Z_{M_3 N^{2/5}}\indic_{B_2}) \ < \ \delta / 4.
$$
Finally we treat the complement $B_3$.
On the complement of $B_1 \cup B_2$ we can find a interval with length $2\epsilon $ contained in $(- 1 / \epsilon, 1 / \epsilon)$, which we denote as $(x-\epsilon,x+\epsilon)$ such that 
\begin{equation}
\label{epsilon^3}
\sum_{y\in (x-\epsilon,x+\epsilon)}\sum_{j=N^{2/5}/2}^{N^{2/5}}\hat{\xi}_j(y)\geq \epsilon^3N^{4/5}.
\end{equation}
$(B_1\cup B_2)^c$ make sure that we have sufficient number of visited sites in $(x-\epsilon,x+\epsilon)$.
Denote 
$$V=\left\{y\in(x-\epsilon,x+\epsilon):\exists~ N^{2/5}/2\leq j\leq N^{2/5},\hat{\xi}_j(y)=1\right\}$$ 
as the set of visited sites between $N^{2/5}/2$ and $N^{2/5}$.
Without loss of generality, we assume that $x=0$. For $y\in (-\epsilon,\epsilon)$, consider a random walk $\{S_i\}_{i\geq0}$ starting from $S_0=y$ and each step it moves to one of its neighbourhood $z$ ($|y-z|\leq N^{-1/5}$) with probability $1/(2N)$. Observe that $\E|S_i-S_{i-1}|^2\approx\frac{1}{3N^{2/5}}$. Let $\tau_{2\epsilon}=\inf\{i>0:S_i\in[-2\epsilon,2\epsilon]^c\}$. 

For any $z\in(-\epsilon,\epsilon)$, there are positive constants $c_1<c_2$ such that
\begin{equation}
\label{eqn:rw_hit}
c_1\epsilon N^{-4/5}\leq \P^y(S_i\text{ hits }z\text{ before }\tau_{2\epsilon})\leq c_2\epsilon N^{-4/5}.
\end{equation} 
Let
$$N_V=\sum_{i=1}^{\tau_{2\epsilon}}\mathbb{I}_{S_i\in V}.$$
Then by (\ref{epsilon^3}) and (\ref{eqn:rw_hit}),
\begin{align*}
\E[N_V] &= \sum_{z\in V}\P^y(S_i\text{ hits }z\text{ before }\tau_{2\epsilon})\\
&\geq\epsilon^3N^{4/5}\cdot c_1\epsilon N^{-4/5}\\
&=c_1\epsilon^4.
\end{align*}
Moreover,
\begin{align*}
\E[N_V^2] &=\sum_{z,z'\in V}\P^y(S_i\text{ hits }z,z'\text{ before }\tau_{2\epsilon})\\
&\leq \E[N_V]+\frac{c_2^2}{c_1^2}(\E[N_V])^2.
\end{align*}
By simple exercise using Cauchy-Schwarz inequality,
\begin{equation}
\label{eqn:positive_visits}
\begin{split}
\P(N_V>0)&\geq \frac{(\E[N_V])^2}{\E[N_V^2]}\\
&\geq\frac{1}{c_2^2c_1^{-2}+(\E[N_V])^{-1}}\\
&\geq\frac{1}{c_2^2c_1^{-2}+(c_1\epsilon^4)^{-1}}.
\end{split}
\end{equation}
For any $y\in(-\epsilon,\epsilon)$, let $\{S_i\}_{i\geq0}$ be the random walk starting from $y$ with the same law above. Define $$\sigma_0=\inf\{i>0:S_i\in [-\epsilon/2,\epsilon/2]\},\sigma'_0=\inf\{i>\sigma_0:S_i\in[-2\epsilon,2\epsilon]^c\},$$
and inductively 
$$\sigma_n=\inf\{i>\sigma'_{n-1}:S_i\in[-\epsilon/2,\epsilon/2]\},\sigma'_n=\inf\{i>\sigma_n:S_i\in[-2\epsilon,2\epsilon]^c\}.$$
We say that $S_i$ visits $n$ times to interval $[-\epsilon/2,\epsilon/2]$ in $M_3N^{2/5}$ steps if $\sigma_n<M_3N^{2/5}<\sigma_{n+1}$. Denote $N_{\epsilon/2}$ to be this number of times of visiting to $[-\epsilon/2,\epsilon/2]$ before $M_3N^{2/5}$. 

Once a particle starting from $(-\epsilon,\epsilon)$ visits $V$, it is killed with probability $\P(N_V>0)$. Hence, each time a particle visit the interval $(-\epsilon/2,\epsilon/2)$, it can survive with probability $1-\P(N_V>0)$. If the particle visits $R$ times to $(-\epsilon/2,\epsilon/2)$, the probability of surviving will be small. If we have a big time horizon $M_3N^{2/5}$, we can make sure that the particle can visit $(-\epsilon/2,\epsilon/2)$ more than $R$ times.

The probability that $\{S_i\}_{i\geq0}$ starting from $y$ does not hit $V$ until $M_3N^{2/5}$ 
\begin{align*}
&~~~~\P^y\left(S_i\notin V\text{ for any }0\leq i\leq M_3-N^{2/5}\right)\\
&\leq \P^y\left(N_{\epsilon/2}\leq R\right)+\P\left(S_i\notin V\text{ for any }0\leq i\leq M_3N^{2/5}\mid N_{\epsilon/2}> R\right).
\end{align*}
By (\ref{eqn:positive_visits}) and the Markov property of $\{S_i\}_{i\geq0}$,
\begin{align*}
\P^y\left(S_i\notin V\text{ for any }0\leq i\leq M_3N^{2/5}\mid N_v\geq R\right)&\leq \left(1-\frac{1}{c_2^{2}c_1^{-2}+(c_1\epsilon^4)^{-1}}\right)^R\\
&\leq \delta/2,
\end{align*}
if $R$ is chosen large enough compared to $\epsilon^{-4}$. 
Moreover,
$$\P^y(N_{\epsilon/2}\leq R)\leq\delta/2,$$
if $M_3$ is chosen largen compared to $R$.
This concludes that
$$\E[\hat{X}_{M_3N^{2/5}}]\leq \delta.$$
\end{proof}

%

\begin{remark}
Notice that without considering the attrition, we can have the probability $\P(T_k<\infty)\leq C2^{-k}$. This is not enough in the proof of Theorem \ref{thm:sub}. However, for the proof we are helped by the attrition: sites that were visited cannot be visited again.
Even in a very small killing zone $(x-\epsilon,x+\epsilon)$ in the proof above, many particles will be killed in a finite but large time period.
\end{remark}

\subsection{The case $\kappa>C_2$}
\label{sec:sup}
In this case, we will prove some properties of the true process, and then lead to an oriented percolation construction.
The first step is to show that the difference between the solution to (\ref{real spde}) and the solution to deterministic heat equation is quite small for small times. Suppose under $\Q$, $u(t,x)$ is the solution to (\ref{real spde}) and under $\P$, $u(t,x)$ is the solution to (\ref{eqn: brw spde}). The Radon-Nykodym derivative of $\Q$ with respect to $\P$ is (\ref{r-n derivative}). Let the initial condition $f$ be continuous, compact supported and $f(x)=1$ for $x\in[-1,1]$. We can regard the initial condition as $\mathbb{I}_{[-1,1]}$ plus some nonsignificant term. Define the difference $$N(t,x)=u(t,x)-G_t f(x),$$ with $G_tf(x)=\E[f(x+B_{t/3})]$, where $(B_t)_{t\geq0}$ is a standard Brownian motion. By Lemma 4.2 of \cite{shiga} (also ref. Lemma 4 of \cite{lalley}),
\begin{align*}
\P\left( | N(t,x) |\geq \sqrt{\delta}\e^{-(\delta^5-t)|x|} \text{ for some }t\leq \delta^5 \text{ and }x\in \R \right)\leq C_1\delta^{-1/12}\exp(-C_2 \delta^{-1/4}).
\end{align*}
This property also holds for $\Q$:
\begin{lemma}
\label{diff real heat}
Denote $$A_\delta=\left\{ | N(t,x) |\leq \sqrt{\delta}\e^{-(\delta^5-t)|x|} \text{ for } \forall t\leq \delta^5 \text{ and }\forall x\in \R \right\}.$$ If under $\Q$, $u(t,x)$ is the solution to (\ref{real spde}) given the initial condition $f$ satisfying $f(x)=1$ for $x\in[-1,1]$, $f(x)=0$ for $x\in[-1-\delta,1+\delta]$ and $f$ is linear in the other parts, then $$\Q(A_{\delta})\geq 1-3\delta^{7/2} \text{ for all }\delta>0 \text{ small enough}.$$
\end{lemma}

\begin{proof}
$A_{\delta}\in \mathcal{F}_{\delta^5}$, hence
\begin{align*}
\Q(A_{\delta}) &= \int_{A_{\delta}}\frac{d\Q}{d\P}d\P \\
&\geq (1-\delta^{7/2})\left( \int_{A_{\delta}\cap \left\{ \frac{d\Q}{d\P}\mid_{\F_{\delta^5}}\geq 1-\delta^{7/2}\right\}}d\P \right)
\end{align*}
Since for $\delta$ small enough, $\P(A_{\delta})\geq 1-\delta^{7/2}$, we only need to show that $$\P\left( \left.\frac{d\Q}{d\P}\right| _{\F_{\delta^5}}\geq 1-\delta^{7/2}\vert A_{\delta} \right)\geq 1-\delta^{7/2}.$$
By (\ref{r-n derivative}),
\begin{align*}
\P \left( \left.\frac{d\Q}{d\P}\right| _{\F_{\delta^5}}\geq 1-\delta^{7/2}\vert A_{\delta} \right) \geq \P \left( \left| \int_0^{\delta^5}\int \theta(t,x)m(dt,dx)+\frac{1}{2}\int_0^{\delta^5}(u_t,\theta(t,\cdot)^2)dt \right|\leq \delta^{7/2} \vert A_{\delta}\right).
\end{align*}
By Chebyshev's inequality,
\begin{align*}
&~~~~\P \left( \left| \int_0^{\delta^5}\int \theta(t,x)m(dt,dx)+\frac{1}{2}\int_0^{\delta^5}(u_t,\theta(t,\cdot)^2)dt \right|\geq \delta^{7/2} \vert A_{\delta}\right)\\
&\leq \frac{2}{\delta^7}\left\{ \E\left( \int_0^{\delta^5}(u_t,\theta(t,\cdot)^2)dt \vert A_{\delta}\right)+\E \left( \left( \int_0^{\delta^5}(u_t,\theta(t,\cdot)^2)dt \right)^2 \vert A_{\delta}\right) \right\}.
\end{align*}
Given $A_{\delta}$,
\begin{align*}
u(t,x)\leq \sqrt{\delta}\e^{-(\delta^5-t)|x|}+\sqrt{\frac{3}{2\pi t}}\int_{-2}^2\e^{-\frac{3|x-y|^2}{2t}}f(y)dy.
\end{align*}
By H\"{o}lder inequality,
\begin{align*}
\E\left( \int_0^{\delta^5}(u_t,\theta(t,\cdot)^2)dt \vert A_{\delta}\right)&\leq \E \left( \int \left(\int_0^{\delta^5}u(t,x)dt\right)^3dx \vert A_{\delta}\right)\\
&\leq \frac{1}{4}\delta^{10} \int \int_0^{\delta^5}\E(u(t,x)^3\vert A_{\delta})dtdx\\
&\leq \frac{1}{4}\delta^{10} \left\{ 3\delta^{3/2}\int \int_0^{\delta^5}\e^{-3(\delta^5-t)|x|}+\frac{3\sqrt{3}}{\sqrt{2\pi t}}\int\int_0^{\delta^5} \int_{-2}^2\e^{-\frac{3|x-y|^2}{2t}}f(y)dydtdx \right\}\\
&\leq \frac{1}{2}\delta^{23/2}+3\delta^{15}.
\end{align*}
Similarly,
\begin{align*}
\E \left( \left( \int_0^{\delta^5}(u_t,\theta(t,\cdot)^2)dt \right)^2 \vert A_{\delta}\right)\leq C\delta^{23}.
\end{align*}
Therefore,
\begin{align*}
\P \left( \left| \int_0^{\delta^5}\int \theta(t,x)m(dt,dx)+\frac{1}{2}\int_0^{\delta^5}(u_t,\theta(t,\cdot)^2)dt \right|\geq \delta^{7/2} \vert A_{\delta}\right)\leq \delta^{9/2},
\end{align*}
and we have the expected result.
\end{proof}
The previous result helps to get a lower bound for the total density in a small time period which is our first desired property.
\begin{corollary} \label{spdedens}
\label{density in period}
Let $f$ be the function given as: $f(x)=1$ for $x\in[-r,r]$ with $r\geq 1$, $f(x)=0$ for $x\in[-r-\delta^{5/2},r+\delta^{5/2}]^c$ and $f$ is linear in the other parts, then there exists constants $L_1(r)<L_2(r)<\infty$,
\begin{align*}
\P \left( \forall x\in \left[-r-2\delta^{5/2},r+2\delta^{5/2}\right],L_1\delta^5\leq\int_{\delta^5/2}^{\delta^5}\hat{u}_t(x)dt\leq L_2\delta^5\right) > 1-\delta^{7/2},
\end{align*}
for all $\delta>0$ small.
\end{corollary}

\begin{proof}
By Lemma \ref{diff real heat}, we know that out of probability $\delta^{7/2}$,
\begin{align*}
\left| \hat{u}_t(x)-G_t f(x) \right|\leq \sqrt{\delta}, \forall t\in [0,\delta^5].
\end{align*}
For any $x\in \left[-r-2\delta^{5/2},r+2\delta^{5/2}\right]$ and any $t\in \left[ \delta^5/2,\delta^5 \right]$, given that $\delta$ is small enough,
\begin{align*}
G_t f(x) &\geq G_{\delta^5}\indic_{[-r,r]}(r+2\delta^{5/2})\\
&\geq 2L_1(r)
\end{align*}
for some constant $L_1(r)>0$.
Hence, for any $x\in \left[-r-2\delta^{5/2},r+2\delta^{5/2}\right]$ and any $t\in \left[ \delta^5/2,\delta^5 \right]$,
$$\hat{u}_t(x)\geq 2L_1.$$
The upper bound $L_2\delta^5$ follows from the same reason as the lower bound.
\end{proof}

In our original percolation model, the edges are not directed. However, it suffices to show percolation in the related model where the vertical edges are directed upward. For this we shall build a block argument, reducing the analysis to that of an oriented percolation model. Here, we keep the notation as in \cite{durrett}. Let
$$\mathcal{L}_0=\{(m,n)\in \Z\times \Z_+:m+n \text{ is even}\}.$$
$\mathcal{L}_0$ is made into a graph by drawing oriented edge from $(m,n)$ to $(m-1,n+1)$ or $(m+1,n+1)$. Random variables $\omega(m,n)\in \{0,1\}$ are to indicate whether $(m,n)$ is open ($\omega(m,n)=1$) or close ($\omega(m,n)=0$). We say that there is a path from $(m,n)$ to $(m',n')$ if there is a sequence of points $x_n=m,\dots,=x_{n'}=m'$ so that $|x_l-x_{l-1}|=1$ for $n<l\leq n'$ and $\omega(x_l,l)=1$ for $n\leq l\leq n'$. Let $$\mathbf{C}_0=\{(m,n):(0,0)\rightarrow (m,n) \}$$
be the cluster containing the origin.

The following steps are to construct the blocks which are considered as sites in the renormalized graph, to define when a renormalized site (block) is open and to define when an edge is open in the renormalized graph. We can then use the comparison theorem in \cite{durrett}. The definition of renormalized sites being open demands a more refined treatment of the approximate density, i.e. one needs to look at a smaller scale, and for those purposes $N^{-3/10}$ is adequate.

\begin{definition}
For a closed interval $I=[a,b]$, $\hat{\xi}$ is said to be $(I,\delta,N)$-good if for the continuous function $f$ satisfying $f(x)= 1$ for $x\in I$, $f(x)=0$ for $x\in[a-\delta,b+\delta]^c$ and $f$ is linear in the other parts,  
$$\sum_{x\in J}\hat{\xi}(x)=\lfloor f(iN^{-3/10})N^{1/10}\rfloor,$$
for any interval $J\subset I$ of the form $[iN^{-3/10},(i+1)N^{-3/10}]$ but for $J\cap I^c\neq\varnothing$, $\sum_{x\in J}\hat{\xi}(x)=0$.
\end{definition}

Corollary \ref{spdedens} and 
Lemma \ref{diff real heat} immediately give the following result for the discrete horizontal process. In the following argument, we take $\delta$ and $N$ so that $\delta^{5}N^{2/5}$ to be an even number.

\begin{corollary} \label{block0}
There exists $\delta_0 > 0$ so that 
given $1 \leq r \leq 2$ and $0<\delta < \delta_0$, 
if $\hat \xi_0 $ is $([-r,r],\delta^{5/2},N)$-good on $\Z/N^{6/5}$, then 
for $N$ large, outside of probability $5 \delta^{7/2}$,
for each $x \in [- r- 2\delta^{5/2}, r+ 2\delta^{5/2}]$, $A\hat{\xi}_k (x) \geq \
L_1/2$ for each $\delta^5 N^{2/5}/2 \leq k \leq \delta^5N^{2/5}$.

\end{corollary}


\begin{definition}
Suppose $\hat{\xi}_0$ is $([a,b],\delta^{5/2},N)$-good. Let a $[a,b]$-subordinated process on certain horizontal layer $\{\tilde{\xi}_k(x)\}_{0\leq k\leq \delta^5N^{2/5}}$ be $\{\hat{\xi}_k(x)\}_{0\leq k\leq \delta^5N^{2/5}}$ killed on $[a-1/2,b+1/2]^c$, i.e. for $0\leq k\leq\delta^5N^{2/5}$
\begin{align*}
\tilde{\xi}_{k+1}(x)=\begin{cases}
1 &\text{ if }\sum_{j\leq k}\tilde\xi_j(x)=0\text{ and }\sum_{y\in\mathcal{N}_k(x)}\tilde\eta_{k+1}(y,x)\geq 1\\
0 &\text{ otherwise},
\end{cases}
\end{align*}
where $\mathcal{N}_k(x)=\{y\sim x:\tilde\xi_k(y)=1\}$ and $\tilde\eta_{k+1}(y,x)=0$ if $x\in[a-1/2,b+1/2]^c$ but over $x,y\in[a-1/2,b+1/2]$, $(\tilde\eta_{k+1}(y,x))_{k,y,x}$ is an i.i.d. sequence of random variables with distribution $\text{Bernoulli}(1/(2N))$.
\end{definition}
Note that this killing property means that no new particles are generated outside $[a-1/2,b+1/2]$ and it is to guarantee an independent structure in the renormalization argument. The $[a,b]$ will not appear when we use the subordinated process since it will always be clear from the context.

\begin{corollary} \label{block1}
There exists $\delta_0 > 0$ so that 
under the conditions of Corollary \ref{block0}, for $0<\delta<\delta_0$ and $N$ large, outside probability $6\delta^{7/2}$,
for each $x \in [- r- 2\delta^{5/2}, r+ 2\delta^{5/2}]$, $A\tilde{\xi}_k (x) \geq \
L_1/2 $ for each $\delta^5N^{2/5}/ 2 \leq k \leq \delta^5N^{2/5}$.

\end{corollary}

\begin{proof}
We suppose $\{\tilde\xi_n\}_{0\leq n\leq \delta^5N^{2/5}}$ is coupled with a true process $\{\hat\xi_n\}_{0\leq n\leq \delta^5N^{2/5}}$ and an envelope process $\{\xi_n\}_{0\leq n\leq \delta^5N^{2/5}}$. For any starting site $z$ such that $\tilde\xi_0(z)=1$, let $\xi^z_n$ be the envelope process with initial condition $\indic_{\{z\}}$. For any $0\leq n\leq\delta^5N^{2/5}$, we have
$$\xi_n(x)=\sum_z\xi^z_n(x),$$
where the sum is over the initial condition that is $([-r,r],\delta^{5/2},N)$-good. The event
\begin{equation}
\label{eqn:tilde process}
\left\{\exists x\in[-r-1/2,r+1/2]\text{ and }0\leq k\leq \delta^5N^{2/5}:\tilde{\xi}_k(x)=0\text{ but }\hat\xi_k(x)=1\right\}
\end{equation}
has probability bounded by
$$\sum_{z}2\P\left(R^z_{\delta^5N^{2/5}}\geq 1/2\right),$$
where the sum is over the initial condition that is $([-r,r],\delta^{5/2},N)$-good and 
$$R^z_n=\max_{m\leq n}\{x\in\Z/N^{6/5}:\xi_m(x)\neq0\}-z$$
is the maximal displacement of $\xi^z$ at time $n$. By (ii) of Lemma \ref{binomial critical} and Kesten's result (\ref{eqn: maximal displacement}), we have
$$\P(R^z_{\delta^5N^{2/5}}>1/2)\leq \frac{4}{\delta^5N^{2/5}}\cdot C_a\delta^{5a}$$
for any $a>0$. Hence the probability of event (\ref{eqn:tilde process}) can be bounded by
$8C_a\delta^{5(a-1)}$ and we can conclude the proof by choosing $a>2$.
\end{proof}

For our block argument the result above provides many sites at level $1$ that are connected to sites occupied by $\hat \xi$ at level 0.  This by itself is insufficient since we require these (level 1) sites to be $([-r-\delta^{5/2},r+\delta^{5/2}],\delta^{5/2},N)$-good.  The following is an important step in this direction.

\begin{lemma}
\label{lem:cum in J}
Let $\tilde \xi_0$ be as in Corollary \ref{block1} and $J$ be a fixed interval of length 
$N^{-3/10}$ in $[-r -2 \delta ^{5/2},r+2 \delta ^{5/2} ]$.  
Then the event that 
\begin{equation} \label{bund}
\min_{x\in[-r -2 \delta ^{5/2},r+2 \delta ^{5/2} ]} \sum_{k= \delta^5N^{2/5}  /2}^{ \delta^5 N^{2/5} } A\tilde{\xi}_k(x) \geq L_1\delta^5 N^{2/5}/4 .
\end{equation}
but 
$$
\sum_{x \in J} \sum_{k= \delta^5 N^{2/5}/2}^{ \delta^5 N^{2/5} } \tilde \xi_k (x) \ < L_1\delta^5 {\sqrt N}/32,
$$
has probability less than $e^{-c \delta^5 \sqrt N}$ for universal $c>0$.
\end{lemma}

\begin{proof}
Let $y$ be the midpoint of $J$ and let
$$
\tau = \inf \left\{ k \geq  \delta^5N^{2/5} /2: \sum_{j=  \delta^5 N^{2/5}/2} ^k A\tilde{\xi}_j(y)
\geq L_1 \delta^5 N^{2/5}/4\right\}.
$$
For $N$ large enough the event $\{ \tau <  \delta^5 N^{2/5}\}$ is contained in event that 
\eqref{bund} happens.  For the proof we note that for every $z$ within $N^{-1/5}$ of $y$ (the range of random walk $\{S_i\}$ starting from $y$), there are at least $N^{9/10} /2$ points of $J$ in $[z-N^{-1/5},z+N^{-1/5}]$ and so while
$\sum_{x \in J} \sum_{k= N^{2/5} \delta^5 /2}^{ N^{2/5} \delta^5 } \tilde \xi_k (x) \ < L_1\delta^5 {\sqrt N}/4$, each such $(z,k)$ pair with $\tilde \xi_k(z) = 1$ represents a probability $$\frac{N^{9/10} /2-L_1\delta^5 {\sqrt N}/4 }{2N}$$ of yielding a fresh occupied site for
$\tilde \xi $ in $J$ at time $k+1$.  The result now follows from standard tail bounds of $\text{Binomial}\left(\frac{L_1\delta^5N^{3/5}}{4},\frac{N^{9/10}/2-\delta^5\sqrt{N}/4}{2N}\right)$.

\end{proof}

Let $\{\tilde\xi_k^i\}_{0\leq k\leq \delta^5N^{2/5}}$ be the subordinate process after killing at level $i\in\mathbb{N}$, where the initial configuration will be recursively defined as indicated at the end of Proposition \ref{prop:1 step} and the subordination effect indicated by the corresponding interval where the configuration is good. 

With the same initial condition, $\{\tilde\xi_k^i\}_{0\leq k\leq \delta^5N^{2/5}}$ follows the same distribution on any vertical level $i$. We will first discuss how the vertical connections behave between layer $0$ and layer $1$ as follows. 
Suppose $\tilde{\xi}^0_0$ is $([-r,r],\delta^{5/2},N)$-good. Until $\delta^5N^{2/5}$ time steps, there is a certain amount of sites $x$'s such that $\tilde{\xi}^0_k(x)=1$. The opening probability of a vertical edge is $\kappa N^{-2/5}$, in the following proposition, we will show that with large probability, the open vertical edge $\langle (x,0),(x,1)\rangle$ can make the initial profile at layer $1$ be $([-r-\delta^{5/2},r+\delta^{5/2}],\delta^{5/2},N])$-good.


\begin{proposition}
\label{prop:1 step}
Given $1 \leq r \leq 2$ and $\delta < \delta_0$, 
there exists vertical connection constant $C_2$, so that for $\kappa>C_2$ and $N$ large enough,
if $\tilde \xi^0_0 $ is $([-r,r],\delta^{5/2},N)$-good on $\Z/N^{6/5} \times \{0 \}$, then 
outside of probability $6 \delta^{7/2}$, on layer $1$,
$\tilde\xi^{1}_0 $ is $([-r-\delta^{5/2},r+\delta^{5/2}],\delta^{5/2},N)$-good on $\Z/N^{6/5}\times\{1\}$.
$\tilde{\xi}^{1}_0 (x) = 1 $ implies that $\tilde{\xi}^0_k (x) = 1 $ for some
$k \in [\delta ^5N^{2/5}/2, \delta^5N^{2/5} ]$ and vertical edge $\langle(x,0),(x,1)\rangle$ is open.
\end{proposition}


\begin{proof}
By Corollary \ref{block1} and Lemma \ref{lem:cum in J}, outside of probability $6\delta^{7/2}$ (for $N$ large enough), we have that for every interval $J=[iN^{-3/10}, (i+1)N^{-3/10})$ contained in $[-r -2 \delta ^{5/2},r+2 \delta ^{5/2} ]$, we have
$$
\sum_{x \in J} \sum_{k= \delta^5N^{2/5}  /2}^{ \delta^5N^{2/5}  } \hat \xi_k (x) \ \geq \  L_1\delta^5 {\sqrt N }/32.
$$

We simply require that the vertical connection constant $C_2$ be greater than $64/(L_1\delta^5)$.
Then by standard tail bounds of $\text{Binomial}\left(L_1\delta^5\sqrt{N}/32,64/(L_1\delta^5)N^{-2/5}\right)$, there exists universal $c>0$ so that 
outside probability $2e^{-cN^{1/10}}N^{3/5}$,
for every such interval $J$, the number of $x \in J$ so that for some $k \in [\delta^5N^{2/5}  /2, \delta^5N^{2/5}  ], \tilde\xi_k(x) =1 $ and $\langle(x,0),(x,1)\rangle$ is open
is greater than $ N^{1/10}$.
\end{proof}

%
Initially, $\tilde{\xi}^0_0$ is $([-1,1],\delta^{5/2},N)$-good. By Proposition \ref{prop:1 step}, with probability $1-6\delta^{7/2}$, $\tilde{\xi}^1_0$ is $([-1-\delta^{5/2},1+\delta^{5/2}],\delta^{5/2},N)$-good. 
We can define recursively $\{\tilde{\xi}^i_k\}_{0\leq k\leq \delta^5N^{2/5}}$ for $0\leq i\leq \delta^{-5/2}$ (from here toward the end, we take $\delta<\delta_0$ and $\delta^{-5/2}\in\mathbb{N}$).  
By FKG inequality, with probability
$$(1-6\delta^{7/2})^{2\delta^{-5/2}}\geq 1-12\delta,$$
$\tilde{\xi}^{2\delta^{-5/2}}_{\delta^5N^{2/5}}$ is $([-3,3],\delta^{5/2},N)$-good. We then split and only consider the particles in two intervals $[-3,-1]$ and $[1,3]$.
We run over two processes $\{\tilde{\xi}^i_k\}_{0\leq k\leq \delta^5N^{2/5}},2\delta^{-5/2}\leq i\leq 4\delta^{-5/2}$ starting from layer $2\delta^{-5/2}$
with initial conditions to be $([-3,-1],\delta^{5/2},N)$-good and $([1,3],\delta^{5/2},N)$-good. Recursively, given $\tilde{\xi}^{2n\delta^{-5/2}}_0$ is $(2m+[-1,1],\delta^{5/2},N)$-good, then outside of probability $12\delta$, $\tilde{\xi}^{2(n+1)\delta^{-5/2}}_0$ is $(2(m+1)+[-1,1],\delta^{5/2},N)$-good and $(2(m-1)+[-1,1],\delta^{5/2},N)$-good. 

Note that the particles from $\tilde{\xi}^{2n\delta^{-5/2}}_0$ with initial conditions $(2(m-1)+[-1,1],\delta^{5/2},N)$-good and  $(2(m+1)+[-1,1],\delta^{5/2},N)$-good will meet in $[-1,1]+2m$ at layer $2(n+1)\delta^{-5/2}$. We will only inherit the particles with lower $m$ index, i.e. the particles from those with initial condition  $(2(m-1)+[-1,1],\delta^{5/2},N)$-good.

Now we can do the renormalization. The renormalizaed regions are defined as $$R_{m,n}=[-4,4]\times [0,2\delta^{-5/2}]+(2m,2n\delta^{-5/2})$$
and $$I_m=[-1,1]+2m.$$
The renormalized site $(m,n)$ corresponds to the block $R_{m,n}$. The random variables $\omega(m,n)\in \{0,1\}$ is to indicate that the renormalized block (site in the renormalized graph) is open or close. $\omega(m,n)=1$ if $\tilde{\xi}^{2n\delta^{-5/2}}_0$ is $(2m+[-1,1],\delta^{5/2},N)$-good in $R_{m,n}$ and we say that $R_{m,n}$ is good. The event that $\omega(m,n)$ is open or not is measurable with respect to the graphical representations in $R_{m,n}$ by the definition of $\{\tilde\xi_k\}_{0\leq k\leq \delta^5N^{2/5}}$ on a certain level. For an edge $e=\langle (m,n),(m+1,n+1) \rangle$ or $e=\langle (m,n),(m-1,n+1) \rangle$, denote $\psi(e)$ as the state of the edge. For $e=\langle (m,n),(m+1,n+1) \rangle$, $\psi(e)=1$ if $(m,n)$ and $(m+1,n+1)$ are open sites in the renormalized graph. The definition of $\psi(e)$ for $e=\langle (m,n),(m-1,n+1) \rangle$ is similar. Let the probability of an edge being open in the renormalized graph be $\P(\psi(e)=1)=1-12\delta$ and $\P(\psi(e)=0)=12\delta$.

Therefore, the renormalized space is $\mathcal{L}_0=\{(m,n)\in \Z^2: m+n \text{ is even }, n\geq 0\}$ and make $\mathcal{L}_0$ into a graph $\mathcal{G}=(\mathcal{V},\mathcal{E})$ by drawing oriented edges from $(m,n)$ to $(m\pm 1,n+1)$.  The percolation process $(\psi(e))_{e\in \mathcal{E}}$ is called $d$-dependent percolation with density $p$ if for a sequence of vertices $v_i=(m_i,n_i), 1\leq i\leq I$ with $\|v_i-v_j\|_{\infty}>d, i\neq j$ connected by a sequence of edges $e_i,1\leq I-1$, $$\P(\psi(e_i)=0,1\leq i\leq I-1)\leq (1-p)^{I-1}.$$

\begin{proposition}
The percolation process $(\psi(e))_{e\in \mathcal{E}}$ is a $1$-dependent oriented percolation with density $1-12\delta$.
\end{proposition}

The initial condition is $\omega(0,0)=1$. By using the comparison argument Theorem 4.3 in \cite{durrett}, we have the following result.

\begin{theorem}
If there exists a percolation in the renormalized space $\mathcal{L}_0$ just defined, then there is a percolation in our anisotropic percolation process.
\end{theorem}

The theorem of existence of percolation for $d$-dependent oriented percolation (Theorem 4.1 in \cite{durrett}) shows that if $12\delta<6^{-4\cdot 9}$, there is a percolation. 
\begin{remark}
Figure \ref{fig:percolation graph} shows this renormalization construction.
\end{remark}

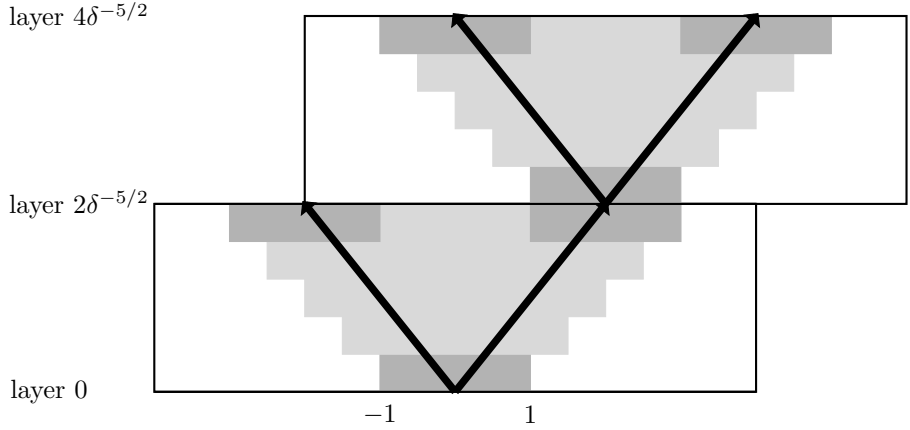
\begin{figure}
\begin{center}
\begin{tikzpicture}
\draw [thick] (-4,0) rectangle (4,2.5);
\node at (-1,-0.3) {$-1$};
\node at (1,-0.3) {$1$};
\draw [fill=gray!60, gray!60] (-1,0) rectangle (1,0.5);
\draw [thick, black]  (-4,0) -- (4,0);
\node at (-5.4,0) {layer $0$};
\draw [fill=gray!30, gray!30] (-1.5,0.5) rectangle (1.5,1);
\draw [fill=gray!30, gray!30] (-2,1) rectangle (2,1.5);
\draw [fill=gray!30, gray!30] (-2.5,1.5) rectangle (2.5,2);
\draw [fill=gray!30, gray!30] (-3,2) rectangle (3,2.5);
\draw [fill=gray!60, gray!60] (1,2) rectangle (3,2.5);
\draw [fill=gray!60, gray!60] (-3,2) rectangle (-1,2.5);
\draw [thick, black] (-4,2.5) -- (4,2.5);
\node at (-5,2.5) {layer $2\delta^{-5/2}$};
\path[draw=black,solid,line width=1mm,fill=black,
preaction={-triangle 90,thin,draw,shorten >=-0.5mm}
]  (0,0) -- (2,2.5);
\path[draw=black,solid,line width=1mm,fill=black,
preaction={-triangle 90,thin,draw,shorten >=-0.5mm}
]  (0,0) -- (-2,2.5);
\draw [fill=gray!60, gray!60] (1,2.5) rectangle (3,3);
\draw [fill=gray!30, gray!30] (0.5,3) rectangle (3.5,3.5);
\draw [fill=gray!30, gray!30] (0,3.5) rectangle (4,4);
\draw [fill=gray!30, gray!30] (-0.5,4) rectangle (4.5,4.5);
\draw [fill=gray!30, gray!30] (-1,4.5) rectangle (5,5);
\draw [fill=gray!60, gray!60] (-1,4.5) rectangle (1,5);
\draw [fill=gray!60, gray!60] (3, 4.5) rectangle (5,5);
\path[draw=black,solid,line width=1mm,fill=black,
preaction={-triangle 90,thin,draw,shorten >=-0.5mm}
]  (2,2.5) -- (4,5);
\path[draw=black,solid,line width=1mm,fill=black,
preaction={-triangle 90,thin,draw,shorten >=-0.5mm}
]  (2,2.5) -- (0,5);
\draw [thick] (-2,2.5) rectangle (6,5);
\node at (-5,5) {layer $4\delta^{-5/2}$};
\end{tikzpicture}
\end{center}
\caption{Oriented percolation construction \label{fig:percolation graph}}
\end{figure}
\newpage
\appendix
\section{Estimations for showing tightness}
\label{appendix}
First, we need some bounds on the distribution function of $S_t$. Recall that $S_t=\sum_{i=1}^t Y_i$, with $(Y_i)$ i.i.d. uniformly distributed on $\{i/N^{1+\alpha},|i|\leq N\}$ and $p(t,x)$ is the transition probability of the standard Brownian motion.
\begin{lemma}
\label{lem: diff}
There exists $m$, such that for $N\geq m$ and any $t\in \mathbb{N}$,
\begin{align}\label{diff}
\left|N^{1+\alpha}\P(S_t=y)-p\left(\frac{c_3t}{3N^{2\alpha}},y\right)\right| \leq CN^{\alpha}t^{-\frac{3}{2}},
\end{align}
where $c_3$ is the constant in Section \ref{green fun} that tends to $1$ as $N\rightarrow\infty$.
\end{lemma}
The proof follows from \cite{mueller,normal approx} and the inversion formula of characteristic function in \cite{durrettbook}.
\begin{proof}
$\E \e^{iuS_t}=\rho^t(u)$, where $\rho(u)=\E \e^{iuY}$.
\begin{align*}
\rho(u) &= \E \e^{iuY}\\
 &=\frac{1}{2N} \left(\sum_{k=1}^N \e^{\frac{iuk}{N^{1+\alpha}}}+\e^{\frac{-iuk}{N^{1+\alpha}}}\right)\\
 &=\frac{1}{N}\left(\sum_{k=1}^N \cos\left(\frac{uk}{N^{1+\alpha}}\right)\right)\\
 &= 1-\frac{c_3u^2}{2!\cdot 3 N^{2\alpha}}+ \frac{c_4u^4}{4!\cdot 5 N^{4\alpha}}r, ~|r|\leq 1
\end{align*}
This directly gives that for $u\leq N^{\alpha}/2$,
$$|\rho(u)|\leq \exp\left(-\frac{c_3u^2}{12N^{2\alpha}}\right),$$
and for $u\geq N^{\alpha}/2$, $|\rho(u)|\leq 23/24$.

Moreover, with the help of Theorem 8.5 of \cite{normal approx}, for $u\leq N^{\alpha}/2$,
\begin{align*}
\left|\rho^t(u)-\exp\left(-\frac{c_3tu^2}{6N^{2\alpha}}\right)\right|\leq Ct^{-1}\exp\left(-\frac{c_3tu^2}{6N^{2\alpha}}\right).
\end{align*}

Follow the inversion formula \cite{durrettbook},
$$N^{1+\alpha}\P(S_t=y)=\frac{1}{2\pi}\int_{-\pi N^{1+\alpha}}^{\pi N^{1+\alpha}}\e^{iuy}\rho^t(u)du,$$
$$p\left(\frac{c_3t}{3N^{2\alpha}},y\right)=\frac{1}{2\pi}\int \e^{iuy}\e^{-\frac{tu^2}{6N^{2\alpha}}}du.$$
The difference satisfies
\begin{align*}
|N^{1+\alpha}\P(S_t=y)-p\left(\frac{c_3t}{3N^{2\alpha}},y\right)| &\leq \frac{1}{\pi}\int_{\pi N^{1+\alpha}}^{\infty}\e^{-\frac{c_3tu^2}{6N^{2\alpha}}}du+\frac{1}{\pi}\int_0^{\pi N^{1+\alpha}}|\rho^t(u)-\e^{-\frac{c_3tu^2}{6N^{2\alpha}}}|du\\
&\leq \frac{1}{\pi}\int_{\pi N^{1+\alpha}}^{\infty}\e^{-\frac{c_3tu^2}{6N^{2\alpha}}}du+\frac{1}{\pi}\int_{N^{\alpha}/2}^{\pi N^{1+\alpha}}|\rho^t(u)|+\e^{-\frac{c_3tu^2}{6N^{2\alpha}}}du\\
&~~~~+\frac{1}{\pi}\int_0^{N^{\alpha}/2} \left|\rho^t(u)-\e^{-\frac{c_3tu^2}{6N^{2\alpha}}}\right| du\\
&\leq \frac{1}{\pi}\int_{N^{\alpha}/2}^{\infty}\e^{-\frac{c_3tu^2}{6N^{2\alpha}}}du+N^{1+\alpha}\left( \frac{23}{24} \right)^t+\frac{1}{\pi t}\int_0^{N^{\alpha}/2}\e^{-\frac{c_3tu^2}{6N^{2\alpha}}}du\\
&\leq Ct^{-1}N^{\alpha}e^{-\frac{c_3t}{24}}+N^{1+\alpha}\e^{-\frac{t}{24}}+CN^{\alpha}t^{-\frac{3}{2}}.
\end{align*}
Therefore, we get the bound
\begin{align*}
\left|N^{1+\alpha}\P(S_t=y)-p\left(\frac{c_3t}{3N^{2\alpha}},y\right)\right| \leq C(N^{1+\alpha}\e^{-\frac{t}{24}}+N^{\alpha}t^{-\frac{3}{2}}).
\end{align*}
Because of (a) in the next lemma and $p(t,x)\leq t^{-1/2}$, we have
\begin{align*}
\left|N^{1+\alpha}\P(S_t=y)-p\left(\frac{c_3t}{3N^{2\alpha}},y\right)\right| \leq CN^{\alpha}t^{-\frac{3}{2}}.
\end{align*}
\end{proof}
With the help of Lemma \ref{lem: diff}, we can get the estimations on $\psi_n$.

\begin{lemma}
\label{psi}
We have the following estimates on $\psi_n^z$:
\begin{enumerate}[label=(\alph*)]
\item $(\psi_k^z,1)=1$, $\|\psi_k^z\|_0\leq CN^{\alpha}, \forall k\geq 0$.
\item $(\e_{\lambda},\psi_{\lfloor tN^{2\alpha}\rfloor}^z)\leq C(\lambda,T)\e_{\lambda}(z)$ for $0\leq t\leq T$.
\item  $\|\psi_{k}^z\|_{\lambda}\leq C(\lambda)\e_{\lambda}(z)N^{\alpha}k^{-\frac{1}{2}}$.
\item For $|x-y|\leq 1$,\\ $\| \psi_k^x-\psi_k^y\|_{\lambda} \leq C(\lambda)\e^{\frac{C(\lambda)k}{N^{2\alpha}}}\e_{\lambda}(x)\left(|x-y|^{\frac{1}{2}}k^{-\frac{1}{2}}N^{\alpha}+N^{\frac{\alpha}{2}}k^{-\frac{3}{4}}\right)$.
\item $\| \psi_k^y-\psi_l^y\|_{\lambda} \leq C(\lambda)\e^{\frac{C(\lambda)k}{N^{2\alpha}}}\e_{\lambda}(y)N^{\frac{\alpha}{2}} \left( |k-l|^{\frac{1}{2}} l^{-\frac{3}{4}}+k^{-\frac{3}{4}}\right)$.
\end{enumerate}
\end{lemma}

\begin{proof}
(a)
$$(\psi_t^z,1)=\frac{1}{N^{1+\alpha}}\sum_x \frac{N^{\alpha}}{2}\indic(x\sim z)=\frac{1}{2N}\sum_x \indic(x\sim z)=1.$$
The second statement is because $\P(X_{t}=y)\leq c/N$ for any $y$.

(b)
 \begin{align*}
(\e_{\lambda},\psi_t^z) &= \frac{1}{N^{1+\alpha}}\sum_x \e_{\lambda}(x)\psi_t^z(x)\\
 &=\sum_x \e_{\lambda}(x)\P(S_{t+1}=x-z)\\
 &\leq 2\e_{\lambda}(z)\sum_x \e^{\lambda x}P(S_{t+1}=x)\\
 &\leq 2\e_{\lambda}(z) (\E \e^{\lambda Y})^{t+1}\\
 &\leq 2\e_{\lambda}(z) \left( 1+\frac{\lambda^2}{3N^{2\alpha}} \right)^{t+1}\\
 &\leq 2\e_{\lambda}(z) \exp\left(\frac{\lambda^2(t+1)}{3N^{2\alpha}}\right).
\end{align*}

(c)
 By (\refeq{diff}) and $p(k,y)\leq Ck^{-1/2}$, we have
\begin{align*}
N^{1+\alpha}\P(S_k=y) &\leq C\left(N^{\alpha}k^{-\frac{1}{2}}+N^{\alpha}k^{-\frac{3}{2}}\right),
\end{align*}
then,
\begin{align*}
\psi_k^0(y)\leq C\left(N^{\alpha}(k+1)^{-\frac{1}{2}}+N^{\alpha}(k+1)^{-\frac{3}{2}}\right).
\end{align*}
Therefore
\begin{align*}
\|\psi_{k}^z\|_{\lambda}\leq C(\lambda,T)\e_{\lambda}(z)N^{\alpha} k^{-\frac{1}{2}}.
\end{align*}

(d)
 For $|x|\geq 1$,
\begin{align*}
\P(S_k=x) &\leq N^{-(1+\alpha)}\P(S_k\geq |x|-1)\\
&\leq N^{-(1+\alpha)}\exp(-u(|x|-1))\E\exp(uX_k)\\
&\leq N^{-(1+\alpha)}\exp(-u(|x|-1))\exp\left(\frac{u^2k}{6N^{2\alpha}}\right).
\end{align*}
Hence, for $|x-z|\geq 1$,
\begin{align*}
\psi^x_k(z)\leq \exp(-u|x-z|)\exp\left(\frac{u^2k}{6N^{2\alpha}}\right).
\end{align*}
This gives that for $|x-z|\geq 2$,
\begin{align*}
\psi_k^x(z)+\psi_k^y(z)\leq \exp(-2 \lambda|x-z|)\exp\left(\frac{2\lambda^2k}{3N^{2\alpha}}\right).
\end{align*}
From (\refeq{diff}) and $|p(t,x)-p(t,y)|\leq Ct^{-1}|x-y|$, we have
\begin{align*}
\| \psi_k^x-\psi_k^y\|_0 \leq C\left( |x-y|N^{2\alpha}k^{-1}+N^{\alpha}k^{-\frac{3}{2}} \right).
\end{align*}
So,
\begin{align*}
\| \psi_k^x-\psi_k^y\|_{\lambda} &\leq \sup_{|x-z|<2}C(\lambda)\| \psi_k^x-\psi_k^y\|_0\e_{\lambda}(z)\\
 &~~~~+\sup_{|x-z|\geq 2}\min \left(\| \psi_k^x-\psi_k^y\|_0,\e^{\frac{C(\lambda)k}{N^{2\alpha}}}\exp(-2\lambda|x-z|)\right)\e_{\lambda}(z)\\
 &\leq C(\lambda)\e^{\frac{C(\lambda)k}{N^{2\alpha}}} \e_{\lambda}(x)\left(\| \psi_k^x-\psi_k^y\|_0+\| \psi_k^x-\psi_k^y\|_0^{\frac{1}{2}}\right)\\
 &\leq C(\lambda)\e^{C(\lambda)k/N^{2\alpha}}\e_{\lambda}(x)\left(|x-y|^{\frac{1}{2}}k^{-\frac{1}{2}}N^{\alpha}+N^{\frac{\alpha}{2}}k^{-\frac{3}{4}}\right).
\end{align*}

(e)
 By using $|p(t,y)-p(s,y)|\leq C|t-s|s^{-3/2}$ and (\refeq{diff}), we have
\begin{align*}
\| \psi_k^y-\psi_l^y\|_0 \leq C\left(|k-l|l^{-\frac{3}{2}}N^{\alpha}+N^{\alpha}k^{-\frac{3}{2}}+N^{\alpha}l^{-\frac{3}{2}}\right).
\end{align*}
Similarly as the argument in (d), we can get
\begin{align*}
\| \psi_k^y-\psi_l^y\|_{\lambda} \leq C(\lambda)\e^{\frac{C(\lambda)k}{N^{2\alpha}}}\e_{\lambda}(y)N^{\frac{\alpha}{2}} \left( |k-l|^{\frac{1}{2}} l^{-\frac{3}{4}}+k^{-\frac{3}{4}}\right).
\end{align*}

\end{proof}

Recall the Burkholder-Davis-Gundy (BDG) inequality for discrete martingale \cite{bdg}:
\begin{align*}
\E(\sup_{1\leq i\leq t}|M_i|^p)\leq C(p)\E\langle M\rangle_t^{\frac{p}{2}},
\end{align*}
where $\langle M\rangle_t=\sum_{k=1}^t\E_{k-1} (d_k^2),d_k=M_k-M_{k-1}$ and $1<p<\infty$. The notation $\E_{k-1}(\cdot)$ means conditional expectation given $\F_{k-1}$.
We have the following moment estimations.
\begin{lemma}
\label{moment}
Suppose the initial condition $A(\xi_0)$ whose linear interpolation converges in $\mathcal{C}$ to $f$ that is continuous and compact supported, then for $T\geq 0$, $p\geq 2$, $\lambda>0$
\begin{enumerate}[label=(\alph*)]
\item $\E \Big( \sup_{k\leq \lfloor tN^{2\alpha}\rfloor}(\nu_k^N,\e_{-\lambda})^p \Big)\leq C(\lambda,p,f,T).$
\item $(\nu_0^N,\psi_t^z)^p\leq C(\lambda,p,f)\e_{\lambda p}(z)$.
\item $\|\E(A^p(\xi_{\lfloor tN^{2\alpha}\rfloor}))\|_{-\lambda p}\leq C(\lambda,p,f,T)$ for $t\leq T$.
\end{enumerate}
\end{lemma}
\begin{proof}
(a)
Plugging $\phi^N_i=\e_{-\lambda}$ into (\ref{decomp}) gives
$$(\nu_n^N,\e_{-\lambda})=(A\xi_0,\e_{-\lambda})+\sum_{i=1}^{n-1}(\nu_i^N,N^{-2\alpha}\Delta_D \e_{-\lambda})+M_n(\e_{-\lambda}).$$
Since $\Delta_D \e_{-\lambda}\leq C(\lambda)\e_{-\lambda}$, thanks to H{\"o}lder inequality, we have
\begin{align*}
\E\Big( \sup_{k\leq \lfloor tN^{2\alpha}\rfloor}(\nu_k^N,\e_{-\lambda})^p\Big) &\leq C(\lambda,p,f)+C(\lambda,p)\E\left( \sum_{i=1}^{\lfloor tN^{2\alpha}\rfloor-1}(\nu_i^N,N^{-2\alpha}\e_{-\lambda}) \right)^p+C(p)\E \sup_{k\leq \lfloor tN^{2\alpha}\rfloor}|M_s(\e_{-\lambda})|^p\\
&\leq C(\lambda,p,f)+C(\lambda,p)t^{p-1}N^{-2\alpha}\sum_{k=1}^{\lfloor tN^{2\alpha}\rfloor-1}\E(\nu_k^N,\e_{-\lambda})^p+C(p)\E\langle M(\e_{-\lambda})\rangle_{\lfloor tN^{2\alpha}\rfloor}^{\frac{p}{2}}
\end{align*}
The square variation in the last term satisfies
\begin{align*}
\langle M(\e_{-\lambda})\rangle_{\lfloor tN^{2\alpha}\rfloor} &\leq \frac{C(\lambda)}{N^{2\alpha}}\sum_{k=1}^{\lfloor tN^{2\alpha}\rfloor} (\nu_{k-1}^N,\e_{-2\lambda})\\
&\leq C(\lambda)\frac{1}{N^{2\alpha}}\sum_{k=1}^{\lfloor tN^{2\alpha}\rfloor}1+(\nu_{k-1}^N,\e_{-\lambda})^2.
\end{align*}
Use H{\"o}lder inequality again, we have
\begin{align*}
\E\Big( \sup_{k\leq \lfloor tN^{2\alpha}\rfloor}(\nu_k^N,\e_{-\lambda})^p\Big) &\leq C(\lambda,p,f)+C(\lambda,p)t^{p-1}N^{-2\alpha}\sum_{k=1}^{\lfloor tN^{2\alpha}\rfloor-1}\E(\nu_k^N,\e_{-\lambda})^p\\
&~~~~+C(p,\lambda)\E\left(\frac{1}{N^{2\alpha}}\sum_{k=1}^{\lfloor tN^{2\alpha}\rfloor}(\nu_{k-1}^N,e_{-\lambda})^2\right)^{\frac{p}{2}}\\
&\leq C(\lambda,p,f)+C(\lambda,p)T^{p-1}N^{-2\alpha}\sum_{k=1}^{\lfloor tN^{2\alpha}\rfloor-1}\E(\nu_k^N,\e_{-\lambda})^p\\
&~~~~+C(\lambda,p)T^{\frac{p}{2}-1}N^{-2\alpha}\sum_{k=1}^{\lfloor tN^{2\alpha}\rfloor-1}\E(\nu_k^N,e_{-\lambda})^p\\
&\leq C(\lambda,p,f,T)+C(\lambda,p,f,T)N^{-2\alpha}\sum_{k=1}^{\lfloor tN^{2\alpha}\rfloor-1}\E(\nu_k^N,e_{-\lambda})^p
\end{align*}
The discrete Gronwall's lemma concludes part (a)
\begin{align*}
\E \Big( \sup_{k\leq \lfloor tN^{2\alpha}\rfloor}(\nu_k^N,\e_{-\lambda})^p \Big)\leq C(\lambda,p,f,T).
\end{align*}

(b)
Let $\overline{\psi}_t^z(x)=N^{1+\alpha}\P(X_t=x-z)$. Since $\psi_t^z(x)=\frac{1}{2N}\sum_{y\sim x}\overline{\psi}_t^z(y)$, $(\nu_n,\psi_t^z)=(A(\xi_n),\overline{\psi}_t^z)$.
It is directly by using (b) of Lemma \ref{psi} since
\begin{align*}
(\nu_0,\psi_t^z)^p &= (A\xi_0,\overline{\psi}_t^z)^p\\
&\leq \|A(\xi_0)\|_{-\lambda}^p(\e_{\lambda},\overline{\psi}_t^z)^p\\
&\leq C(\lambda,p,f)\e_{\lambda p}(z)
\end{align*}

(c)
By (\ref{approx}) and (b),
\begin{align*}
\|\E(A^p(\xi_{\lfloor tN^{2\alpha}\rfloor}))\|_{-\lambda p}\leq C(\lambda,p,f,T)+C(p)\| \E |M_{\lfloor tN^{2\alpha}\rfloor}(\psi_{\lfloor tN^{2\alpha}\rfloor-\cdot})|^p\|_{-\lambda p}.
\end{align*}
For the second term above, by BDG inequality,
\begin{align*}
\E \left|M_{\lfloor tN^{2\alpha}\rfloor}(\psi_{\lfloor tN^{2\alpha}\rfloor-\cdot})\right|^p &\leq \E \langle M(\psi_{\lfloor tN^{2\alpha}\rfloor-\cdot})\rangle_{\lfloor tN^{2\alpha}\rfloor}^{\frac{p}{2}}\\
&\leq \E\left( \sum_{k=1}^{\lfloor tN^{2\alpha}\rfloor}\frac{\|\psi_{\lfloor tN^{2\alpha}\rfloor-k+1}\|}{N^{2\alpha}}(A\xi_{k-1},\psi_{\lfloor tN^{2\alpha}\rfloor-k+1}) \right)^{\frac{p}{2}}\\
&\leq \E\left( \sum_{k=1}^{\lfloor tN^{2\alpha}\rfloor}\frac{|\lfloor tN^{2\alpha}\rfloor-k+1|^{-\frac{1}{2}}}{N^{\alpha}}(A\xi_{k-1},\psi_{\lfloor tN^{2\alpha}\rfloor-k+1}) \right)^{p/2} \text{ (Lemma \ref{psi} (c))}\\
&\leq C(p,T)\sum_{k=1}^{\lfloor tN^{2\alpha}\rfloor}\frac{|\lfloor tN^{2\alpha}\rfloor-k+1|^{-\frac{1}{2}}}{N^{\alpha}}\left(\E A^{\frac{p}{2}}(\xi_{k-1}),\psi_{\lfloor tN^{2\alpha}\rfloor-k+1}\right)\text{ (Lemma \ref{psi} (a))}\\
&\leq C(p,T)\sum_{k=1}^{\lfloor tN^{2\alpha}\rfloor}\frac{|\lfloor tN^{2\alpha}\rfloor-k+1|^{-\frac{1}{2}}}{N^{\alpha}}\| \E A^{\frac{p}{2}}(\xi_{k-1})\|_{-\lambda p}\left(\e_{\lambda p},\psi_{\lfloor tN^{2\alpha}\rfloor-k+1}\right)\\
&\leq C(p,T,\lambda)\e_{\lambda p}(z)\sum_{k=1}^{\lfloor tN^{2\alpha}\rfloor}\frac{|\lfloor tN^{2\alpha}\rfloor-k+1|^{-\frac{1}{2}}}{N^{\alpha}} \| \E A^p(\xi_{k-1})+1\|_{-\lambda p} \text{ (Lemma \ref{psi}(b))}\\
&\leq C(p,T,\lambda)\e_{\lambda p}(z)\left( 1+ \sum_{k=1}^{\lfloor tN^{2\alpha}\rfloor}\frac{|\lfloor tN^{2\alpha}\rfloor-k+1|^{-\frac{1}{2}}}{N^{\alpha}}\| \E(A^p(\xi_{k-1}))\|_{-\lambda p}\right).
\end{align*}
This gives that
\begin{align*}
\|\E(A^p(\xi_{\lfloor tN^{2\alpha}\rfloor}))\|_{-\lambda p}\leq C(\lambda,p,f,T)\left( 1+\sum_{k=1}^{\lfloor tN^{2\alpha}\rfloor}\frac{|\lfloor tN^{2\alpha}\rfloor-k+1|^{-\frac{1}{2}}}{N^{\alpha}}\| \E(A^p(\xi_{k-1}))\|_{-\lambda p} \right).
\end{align*}
The discrete Gronwall lemma completes this proof.

\end{proof}


\end{document}